\documentclass{article}

\usepackage{amsmath,amssymb,amsthm}
\usepackage{tikz-cd}
\usepackage{enumerate}
\usepackage{mathtools}
\usepackage{appendix}

\newtheorem{lem}{Lemma}[section]

\newtheorem{rem}[lem]{Remark}
\newtheorem{eg}[lem]{Example}
\newtheorem{thm}[lem]{Theorem}
\newtheorem{defn}[lem]{Definition}
\newtheorem{cor}[lem]{Corollary}
\newtheorem{prop}[lem]{Proposition}
\newtheorem{deflem}[lem]{Definition--Lemma}
\def\BbK{\mathbb{K}}
\def\G{\mathbb{G}}

\def\cG{\mathcal{G}}

\def\hookar{\ar@{^{(}->}}

\def\fuk{\EuF}
\def\bc{\mathsf{bc}}

\def\cY{\mathcal{Y}}

\def\mir0{F_0}

\newcommand{\fm}{\mathfrak{m}}

\def\cM{\mathcal{M}}
\def\G{\mathbb{G}}

\def\and{\, \& \,}

\def\cM{\mathcal{M}}

\def\nov{r}

\def\cO{\mathcal{O}}

\def\cF{\mathcal{F}}
\def\cA{\mathcal{A}}

\def\cC{\mathcal{C}}
\def\cD{\mathcal{D}}
\def\cR{\mathcal{R}}

\def\NE{\mathrm{NE}}

\def\NE{\mathrm{NE}}
\def\R{\mathbb{R}}
\def\Z{\mathbb{Z}}
\def\fuk{\mathcal{F}}

\def\id{\mathrm{id}}

\def\bc{{\operatorname{bc}}}

\def\prebc{{pre\text{-}bc}}

\newcommand{\bimod}[2]{{{#1}\text{-}mod\text{-}{#2}}}

\def\cN{\mathcal{N}}
\def\Rbar{\overline{\mathcal{R}}}

\def\Cbar{\overline{\mathcal{C}}}

\def\C{\mathbb{C}}

\def\Cbar{\overline{\cC}}
\def\stab{\mathrm{stab}}

\def\Q{\mathbb{Q}}
\def\big{{\operatorname{big}}}
\def\sm{{\operatorname{sm}}}
\def\top{{\mathsf{top}}}

\def\im{\mathrm{im}}

\def\bL{\mathbf{L}}

\def\ex{\mathrm{ex}}

\def\signn{{n(n+1)/2}}
\def\sbc{{\operatorname{sbc}}}
\def\bulk{{\operatorname{bulk}}}
\def\novb{r}
\def\novblam{t}

\title{The cyclic open--closed map and  variations of Hodge structures}
\author{Sheel Ganatra and Nick Sheridan}

\begin{document}

\maketitle

\begin{abstract}
    We construct the cyclic open--closed map for the big (i.e., bulk-deformed) relative Fukaya category, in the semipositive case, and show that it is a morphism of `polarized variations of semi-infinite Hodge structures'. 
    We also give a natural criterion for the map to be an isomorphism, which is verified for example in the context of Batyrev mirror pairs. We conclude in such Calabi-Yau cases that the rational Gromov--Witten invariants can be extracted from the relative Fukaya category, and hence that enumerative mirror symmetry is a consequence of homological mirror symmetry for Calabi-Yau mirror pairs.
\end{abstract}

\section{Introduction}

This paper establishes a structural relationship  between two variations of semi-infinite Hodge structure (VSHS) associated to the symplectic geometry of a smooth projective variety $X$. The first, the ``categorical VSHS of the Fukaya category'' is associated to its (big, or bulk-deformed) Fukaya category via a categorical construction as reviewed in \cite{Sheridan_formulae}. The second, the ``$A$-model VSHS,'' is naturally associated to its (big) quantum cohomology via rational curve counting and topology, see \cite{Barannikov} and \cite[Section 2.3.1]{CIT}. 

The underlying cohomology groups associated to these VSHS are, respectively, the negative cyclic homology of the big Fukaya category and the (quantum) cohomology of $X$. Our paper first constructs, under suitable semi-positivity hypotheses and in a technical framework for Fukaya categories described below, a version of the cyclic open-closed map \cite{Ganatra:cyclic} between these groups. Our main result shows that the cyclic open--closed map is a morphism of VSHS, i.e., it intertwines the structural data associated to VSHS on either side. Specifically, it intertwines higher residue pairings, and it intertwines the Getzler--Gauss--Manin connection with the Dubrovin--Givental connection, as was announced in \cite{Ganatra2015}.

We also give a criterion for the cyclic open--closed map to be an isomorphism (analogous criteria were stated in \cite{Ganatra2015}), which in particular holds in cases where homological mirror symmetry is established. In such cases we learn immediately that the big Fukaya category `knows' the $A$-model VSHS. Under Calabi-Yau hypotheses this implies we can extract the rational Gromov--Witten invariants of $X$ from its Fukaya category and deduce genus 0 enumerative mirror predictions from homological mirror symmetry, facts we explain following the argument of \cite{Ganatra2015}. 
Whereas that reference focused on the `small' quantum VSHS (which suffices to reconstruct the rational Gromov--Witten invariants of Calabi--Yau threefolds, but not higher-dimensional spaces), our results hold for the `big' version, and in particular allow us to reconstruct the Frobenius manifold structure on quantum cohomology from the Fukaya category. See Section \ref{sec:frob} for more details.

The geometric and technical context in which our work establishes these results is that of the (big) {\em relative Fukaya category} of a suitably semi-positive symplectic manifold. Our constructions build on the technical foundations for such categories and their open-closed maps which were developed in the series of papers \cite{perutz2022constructing,relfukii}. 
In the first paper of that series \cite{perutz2022constructing}, the `small' relative Fukaya category of a symplectic manifold relative to a divisor was constructed.  In the second \cite{relfukii}, this was generalized to the `big' relative Fukaya category, the closed--open and open--closed maps were constructed, their basic properties established, Abouzaid's split-generation criterion was established, and an efficient formalism for constructing Floer-theoretic operations on such categories was set up which we will use.
Our paper, combined with these results \cite{perutz2022constructing,relfukii}, establishes the technical foundations of the relative Fukaya category outlined in \cite[Section 4]{Ganatra2015}; \cite[Section 2.5]{sheridan2021homological}; \cite[Theorems B and C]{Ganatra2023integrality}; and \cite[Theorem B]{GHHPS}. 
In particular, the results stated there as being contingent on foundational results about the relative Fukaya category, are now proved unconditionally. 

In the following subsections we give a more detailed overview of the key definitions, our main results, and applications to Calabi-Yau mirror symmetry.

\subsection{Big relative Fukaya category}

We summarize the definition of the big relative Fukaya category given in \cite{relfukii} (see also \cite{perutz2022constructing}). 
It depends on the following geometric data:
\begin{itemize}
    \item a compact $2n$-dimensional symplectic manifold $(X,\omega)$;
    \item a Liouville subdomain $(W,\theta) \subset X$ (so $d\theta = \omega|_W$);
    \item a grading datum $\G$, which comes equipped with a morphism $H_1(\cG^{or}(W)) \to \G$, where $\cG^{or}(W)$ denotes the Grassmannian of oriented Lagrangian subspaces of $W$;
    \item a `system of divisors' $V = \cup_{q \in Q} V_q \subset X \setminus W$;
    \item an $\omega$-compatible almost complex structure $J_0$ on $X$ such that each component $V_q$ of $V$ is a $J_0$-holomorphic submanifold, and there exists a convex collar for $W$;
    \item a Morse--Smale pair $(f,g)$ on $X$;
    \item certain choices of `perturbation data' for the relevant moduli spaces of pseudoholomorphic curves and Morse flowlines. 
\end{itemize}

We define $\kappa \in H^2(X,W;\R)$ to be the relative cohomology class determined by $\omega$ and $\theta$, and we define $Nef \subset H^2(X,W;\R)$ to be the convex cone spanned by the classes $PD([V_q])$. 
The data are required to satisfy the following conditions:
\begin{itemize}
    \item \textbf{(Ample)} $\kappa$ lies in the interior of $Nef$.
    \item \textbf{(Semipositive)} There exists a class $\tilde{c}_1 \in Nef$ which is a lift of $c_1(TX) \in H^2(X)$ along $H^2(X,W) \to H^2(X)$. 
\end{itemize}

\begin{rem}\label{rem:ag data} 
We repeat here \cite[Remark 1.3]{relfukii}, for the reader's convenience. Following \cite[Section 1.2]{perutz2022constructing}, we may construct the geometric data data $(W \subset X, \omega, \theta, V, J_0)$ from algebro-geometric data as follows. Let $X$ be a smooth complex projective variety, $D \subset X$ a simple normal crossings divisor with components indexed by $P$, and $Nef \subset Div(X, D)_\R \cong \R^P$ a rational polyhedral cone
in the space of divisors supported on $D$ such that:
\begin{itemize}
    \item $N ef$ contains an ample class in its interior;
    \item $N ef$ is contained in the cone of effective semiample divisors supported on $D$;
    \item $Nef$ contains a divisor homologous to the anticanonical divisor modulo torsion.
\end{itemize}
Then we can construct a non-empty path-connected set of data $(W \subset X, \omega, \theta, J_0)$ as above, where $W \subset X$ is a deformation retract of $X \setminus D$, $\omega$ is a K\"ahler form on $X$, $\kappa = [\omega; \theta] \in H^2(X, W ; \R) \cong Div(X, D)_\R$ is an ample class in the interior
of $Nef$, and $J_0$ is the integrable complex structure (see \cite[Section 9.1]{perutz2022constructing}). We
may also construct a system of divisors $V$ such that the classes $P D([V_q ])$ span $N ef$, see \cite[Section 9.2]{perutz2022constructing}. 
\end{rem}

In light of Remark \ref{rem:ag data}, we denote the big relative Fukaya category associated to geometric data as above by $\fuk^\big(X,D)$, where $D$ is to be thought of as some divisor giving rise to the geometric data via Remark \ref{rem:ag data} (even though in general we do not require that any such divisor exists).  
We will similarly denote $\fuk(X \setminus D) := \fuk(W)$.

Objects of $\fuk^\big(X,D)$ are compact exact Lagrangians $L \subset W$ equipped with a $\G$-grading (which means a lift of the canonical map $L \to \cG(W)$ to the abelian cover $\tilde \cG \to \cG^{or}(W) \to \cG(W)$ corresponding to $\pi_1(\cG^{or}(W)) \to H_1(\cG^{or}(W)) \to \G$) and spin structure.

The coefficient ring of the small relative Fukaya category is the $\G$-graded $\fm^\sm$-adic completion of the group ring $\Z[\NE]$, where $\NE \subset H_2(X,W)$ is the monoid of integral classes in the dual cone to $Nef$, and $\fm^\sm$ is the ideal generated by $\nov^u$ for non-zero $u \in \NE$. 
To define the coordinate ring of the big relative Fukaya category, we choose a basis $\alpha^i$ for the Morse cochain complex $CM^*(f,g)$ associated to the Morse--Smale pair $(f,g)$. 
The coefficient ring for the big relative Fukaya category is then the ring of `divided power series' $R^\big \subset \Q \hat\otimes R^\sm[[\novb_i]]$, where the bulk variable $\novb_i$ has degree $|\novb_i| = 2-|\alpha^i|$. 
More precisely, it is the completion of the $R^\sm$-subalgebra generated by the classes $\nov_i^k/k!$, for the filtration induced by the intersection of of the $\fm^\big$-adic filtration with this subalgebra, where $\fm^\big$ is generated by $\fm^\sm$ and the bulk variables. 
The coefficient ring $R^\big$ is equipped with a natural $\G$-grading, a filtration, and a differential (coming from the Morse differential). 

Morphism spaces in $\fuk^\big(X,D)$ are free $R^\big$-modules of finite rank, generated by intersection points between Hamiltonian perturbations of the respective Lagrangians. 
The structure maps count pseudoholomorphic discs $u$, weighted by a monomial in the bulk variables $\nov_i$ which records the constraints on Morse flowlines which are imposed on $u$, as well as a monomial $\nov^u$ recording the homology class $[u] \in \NE$. 

\subsection{Bounding cochains on the small category}

For expository reasons, in the statements of our main results we will restrict our choice of coefficient rings to the most geometrically interesting one, by base-changing to a universal Novikov ring over $\C$. 
Analogues of our main results over more general coefficient rings can easily be deduced from the chain-level statements proven in the body of the paper, which are stated in maximal generality.

In this section we recall the construction of the category of bounding cochains on the small relative Fukaya category over the Novikov field, following \cite[Section 1.4]{perutz2022constructing} and \cite[Section 1.3]{relfukii} (where slightly different notation was used, for consistency with \cite{perutz2022constructing,GHHPS}). 

We assume that $R^\sm$ is concentrated in degree $0 \in \G$. 
This is the case, for example, if $\G = \Z$ and the morphism $H_1(\cG^{or}) \to \Z$ is induced by a holomorphic volume form on $X$; it is also the case unconditionally if $\G = \Z/2$. 
Let $K \subset \R$ be a subgroup containing the image of $\kappa:H_2(X,W) \to \R$. 
For example, if the class $\kappa$ is integral, we may take $K = \Z$.

\begin{defn}
We define the \emph{Novikov field}:
$$\Lambda^\sm := \left\{ \sum_{i=0}^\infty c_i T^{k_i}: c_i \in \C, k_i \in K, \lim_{i \to \infty} k_i = +\infty\right\},$$
and equip it with the trivial filtration. 
We define the \emph{Novikov ring} $\Lambda^\sm_{\ge 0} \subset \Lambda^\sm$, consisting of sums with all $k_i \ge 0$, and equip it with the $\fm^\sm_\Lambda$-adic filtration, where $\fm^\sm_\Lambda$ is the ideal generated by $T^A$, where $A = \min\{\kappa(\nov^u): u \in \NE \setminus \{0\}\}$.
\end{defn}

There is a filtered ring homomorphism
\begin{align*}
    R^\sm & \to \Lambda^\sm_{\ge 0} \qquad \text{sending}\\
    \nov^u & \mapsto T^{\kappa(u)}.
\end{align*}
We define
\begin{itemize}
    \item $\fuk^\sm(X,D;\Lambda^\sm_{\ge 0}):= \fuk^\sm(X,D) \otimes_{R^\sm} \Lambda^\sm_{\ge 0}$ by base-changing along the above filtered ring homomorphism;
    \item $\fuk^\sm(X,D;\Lambda^\sm_{\ge 0})^\bc$ by forming the category of bounding cochains, in accordance with \cite[Definition 2.15]{relfukii}; this implicitly involves changing the filtration on the coefficient ring $\Lambda^\sm_{\ge 0}$ to the trivial filtration;
    \item $\fuk^\sm(X,D;\Lambda^\sm)^\bc$ by base-changing to $\Lambda^\sm$ (this is well-defined as the inclusion $\Lambda^\sm_{\ge 0} \to \Lambda^\sm$ is filtered, once the filtration on $\Lambda^\sm_{\ge 0}$ has been changed to the trivial one). 
\end{itemize}
The category $\fuk^\sm(X,D;\Lambda^\sm)^\bc$ is a $\Lambda^\sm$-linear $A_\infty$ category. 
In particular, as the filtration on $\Lambda^\sm$ is trivial, the category is uncurved (cf. \cite[Section 2.4]{relfukii}). 

\subsection{`Small' bounding cochains on the big category}

We now define the category of bounding cochains over the big relative Fukaya category to which our main results will apply; the definition is not quite the standard one.

\begin{defn}
Choose a basis $\gamma^i$ for $H^*(X;\Q)$. 
We define $\Lambda^\big_{\ge 0}$ to be the $\G$-graded $\fm^\big_\Lambda$-adic completion of $\Lambda^\sm_{\ge 0}[\novblam_i]$, where the bulk variable $\novblam_i$ has degree $2-|\gamma^i|$; $\fm^\big_\Lambda$ is the ideal generated by $\fm^\sm_\Lambda$ and the bulk variables $\novblam_i$. 

We define $\Lambda^\big$ to be the $\G$-graded $\fm^{\bulk}_\Lambda$-adic completion of $\Lambda^\sm[\novblam_i]$, where $\fm^{\bulk}_\Lambda$ is the ideal generated by the bulk variables $\novblam_i$. 
\end{defn}

For each basis element $\gamma^i$ of $H^*(X;\Q)$, we choose a chain-level representative in $CM^*(f,g)$:
$$\gamma^i = \left[ \sum_j C^i_j \alpha^j \right].$$
This allows us to define a filtered ring homomorphism
\begin{equation}\label{eq:RbigLambig}
R^\big \to \Lambda^\big_{\ge 0}
\end{equation}
which extends the morphism $R^\sm \to \Lambda^\sm_{\ge 0}$ and sends
$$\novb_j \mapsto \sum_i C^i_j \novblam_i.$$
(More abstractly, we choose a chain map $(H^*(X;\Q),0) \to (CM^*(f,g),d_{f,g})$ which induces the identity on cohomology, then its dual induces the map $R^\big \to \Lambda^\big_{\ge 0}$ on the bulk variables.)

\begin{deflem}[See Section \ref{sec:sbc}]\label{deflem:bigsmall}
    There is a $\Lambda^\big$-linear $A_\infty$ category $\fuk^\big(X,D;\Lambda^\big)^\sbc$, which is a deformation of $\fuk^\sm(X,D;\Lambda^\sm)^\bc$ over $\Lambda^\big$. 
    Explicitly, this means that the categories have the same set of objects, the morphism spaces are obtained by base-change along the morphism $\Lambda^\sm \to \Lambda^\big$, and the $A_\infty$ structure maps of $\fuk^\big(X,D;\Lambda^\big)^\sbc$ are equal to those of $\fuk^\sm(X,D;\Lambda^\sm)^\bc$ modulo $\fm^{\bulk}_\Lambda$.
\end{deflem}

The notation `$\sbc$' stands for `small bounding cochains'.

The content of Definition--Lemma \ref{deflem:bigsmall} is purely algebraic. 
Briefly, the objects of the category are pairs $(L,b)$ where $L$ is a Lagrangian brane and $b$ is a bounding cochain over $\Lambda^\sm_{\ge 0}$. 
We change coefficients of the big Fukaya category via the homomorphism \eqref{eq:RbigLambig}, then deform the resulting $A_\infty$ structure maps by the bounding cochains $b$, just as in the usual definition; however, as the Maurer--Cartan equation is only satisfied modulo $\fm^{\bulk}_\Lambda$ (the ideal generated by the bulk variables), the curvature of each object is not necessarily zero, but will only lie in $\fm^{\bulk}_\Lambda$. 
This is precisely the condition required of a $\Lambda^\big$-linear $A_\infty$ category, in accordance with \cite[Definition 2.10]{relfukii}. 

\begin{rem}
    One could also define a different coefficient ring $\bar\Lambda^\big$, which is equal to $\Lambda^\big$ as a ring but has instead the trivial filtration; then one could define a $\bar\Lambda^\big$-linear $A_\infty$ category by first base-changing to $\Lambda^\big_{\ge 0}$, then passing to bounding cochains in the sense of \cite[Definition 2.15]{relfukii}, then base-changing to $\bar\Lambda^\big$. 
    As the filtration on $\bar\Lambda^\big$ is trivial, this would produce an uncurved $A_\infty$ category which we would denote $\fuk^\big(X,D;\bar\Lambda^\big)^\bc$, and is an instance of the general construction outlined in \cite[Section 1.3]{relfukii}. 
    However, in general, this category need have no non-trivial objects, even in situations where $\fuk^\sm(X,D;\Lambda^\sm)^\bc$ is well-behaved: obstructions may appear in the Maurer--Cartan equation at higher order in the bulk variables. 
    In contrast, any object of $\fuk^\sm(X,D;\Lambda^\sm)^\bc$ defines an object of $\fuk^\big(X,D;\Lambda^\big)^\sbc$, and we know situations where $\fuk^\sm(X,D;\Lambda^\sm)^\bc$ is well-behaved (e.g., in situations where homological mirror symmetry statements hold for it \cite{sheridan2015homological,sheridan2021homological,Ganatra2023integrality,GHHPS}, we may show that it is homologically smooth). 
    This will be crucial to prove one of our main results, Theorem \ref{thm:vshs}, which gives a criterion under which we can extract the big quantum VSHS from the Fukaya category. 
    To prove this result, we first give a natural criterion for the small open--closed map to be an isomorphism (Theorem \ref{thm:aut_gen}; the criterion is known to be satisfied in the above-mentioned cases where homological mirror symmetry has been proved); then bootstrap from this to prove that the `big' cyclic open--closed map is an isomorphism (Lemma \ref{lem:OCiso}). 
\end{rem}

\subsection{Main results}

Before stating our main results, we recall some notation: $QH^*$ denotes quantum cohomology, defined as in \cite[Section 4.1]{relfukii}; $HH^*$ denotes Hochschild cohomology, defined as in \cite[Section 2.7]{relfukii}; $HH_*$ (respectively, $fHH_*$) denotes Hochschild homology (respectively, filtered Hochschild homology), defined as in \cite[Section 2.8]{relfukii}; $HC_*^-$ (respectively, $fHC_*^-$) denotes negative cyclic homology (respectively, filtered negative cyclic homology), defined as in Section \ref{sec:cyc}.

We now state our main results for the small Fukaya category, some of which are special cases of results in \cite{relfukii}. 
We will need to repeat the results for the big Fukaya category, but there are important distinctions which would get lost if we only stated the `big' versions.

\begin{thm}[Theorem 1.10 of \cite{relfukii}]\label{thm:co_int}
There is a unital graded $\Lambda^\big$-algebra homomorphism
$$\cC\cO_\Lambda^\big: QH^*(X;\Lambda^\big) \to HH^*(\fuk^\big(X;D;\Lambda^\big)^\sbc)$$
called the \emph{big closed--open map}. 
Quotienting by $\fm^\bulk_\Lambda$, one gets a unital graded $\Lambda^\sm$-algebra homomorphism
$$\cC\cO_\Lambda^\sm: QH^*(X;\Lambda^\sm) \to HH^*(\fuk^\sm(X;D;\Lambda^\sm)^\bc)$$
called the \emph{small closed--open map}. 
The small closed--open map coincides with the first-order deformation class of $\fuk^\big(X,D;\Lambda^\big)^\sbc$, in the direction of the bulk variables.
\end{thm}
\begin{proof}
    The proof is the same as that of \cite[Theorem 1.10]{relfukii}, in the case $S = \Lambda^\big$: the only difference is that we pass to small bounding cochains, rather than bounding cochains, but the same proof applies verbatim.
\end{proof}

\begin{cor}\label{cor:vers}
    If $\cC\cO_\Lambda^\sm$ is an isomorphism, then $\fuk^\big(X,D;\Lambda^\big)^\sbc$ is a versal deformation of $\fuk^\sm(X,D;\Lambda^\sm)^\bc$. 
\end{cor}
\begin{proof}
    This follows immediately from the fact that $\cC\cO_\Lambda^\sm$ coincides with the first-order deformation class of $\fuk^\big(X,D;\Lambda^\big)^\sbc$, together with a standard deformation theory argument (see, e.g., \cite[Theorem 3.3]{Sheridan_CDM}). 
\end{proof}

\begin{thm}[Theorem 1.11 of \cite{relfukii}] \label{thm:oc_int}
There is a $QH^*(X;\Lambda^\big)$-module homomorphism
$$\cO\cC_\Lambda^\big: fHH_*(\fuk^\big(X,D;\Lambda^\big)^\sbc)[-n] \to QH^*(X;\Lambda^\big)
$$
called the \emph{big open--closed map}.\footnote{We recall that `$fCC_*$' stands for the (adically) completed Hochschild chain complex, and `$fHH_*$' for its cohomology, cf. \cite[Section 2.8]{relfukii}.}
Quotienting by $\fm^\bulk_\Lambda$, one gets a $QH^*(X;\Lambda^\sm)$-module homomorphism
$$\cO\cC_\Lambda^\sm: HH_*(\fuk^\sm(X,D;\Lambda^\sm)^\bc)[-n] \to QH^*(X;\Lambda^\sm)
$$
called the \emph{small open--closed map}.
\end{thm}
\begin{proof}
     As with Theorem \ref{thm:co_int}, the proof is the same as that of \cite[Theorem 1.11]{relfukii}, except that we use small bounding cochains rather than bounding cochains.
\end{proof}

\begin{thm}\label{thm:mukai_int}
    The (big or small) open--closed map respects pairings, in the sense that
    $$\langle \cO\cC_\Lambda(\alpha),\cO\cC_\Lambda(\beta)\rangle = (-1)^\signn \langle \alpha, \beta \rangle_{Muk}$$
    where $\langle -,-\rangle$ is the pairing on quantum cohomology (with conventions as in \cite[Section 4.1]{relfukii}), and $\langle -,-\rangle_{Muk}$ is the Mukai pairing on Hochschild homology (with conventions as in \cite[Definition 5.19]{Sheridan_formulae}). 
\end{thm}

Theorem \ref{thm:mukai_int} will be proved in Section \ref{sec:oc pair}. A similar result is proved, in the positively monotone case, in \cite[Appendix B]{Ganatra2016}.

\begin{thm}\label{thm:wpcy}
    The map
    \begin{align*}
        \varphi: HH_{-n}(\fuk^\sm(X,D;\Lambda^\sm)^\bc) & \to \Lambda^\sm\\
        \varphi(\alpha) & := \langle \cO\cC_\Lambda^\sm(\alpha),e\rangle
    \end{align*}
    defines a \emph{weak proper Calabi--Yau structure} (see, e.g., \cite[Definition A.2]{Sheridan2013} for the definition). 
\end{thm}

Theorem \ref{thm:wpcy} will be proved in Section \ref{sec:wpcy}. A similar result is proved, in the positively monotone case, in \cite[Lemma 2.4]{Sheridan2013}.

\begin{cor}[Proposition 2.6 of \cite{Sheridan2013}] \label{cor:co_oc_dual}
    The following diagram commutes:
    $$\begin{tikzcd}
         QH^*(X;\Lambda^\sm) \ar[rr,"\alpha \mapsto \langle \alpha{,}\, -\rangle"] \ar[d,"\cC\cO_\Lambda^\sm"]   & &QH^*(X)^\vee[-2n] \ar[d,"(\cO\cC_\Lambda^\sm)^\vee"]  \\
        HH^*(\fuk^\sm(X,D;\Lambda^\sm)^\bc) \ar[rr,"\alpha \mapsto \varphi(\alpha \cap -)"] && HH_*(\fuk^\sm(X,D;\Lambda^\sm)^\bc)^\vee[-n],
     \end{tikzcd}
    $$
    where the horizontal arrows are isomorphisms; in particular, $\cC\cO_\Lambda$ and $\cO\cC_\Lambda$ are dual, up to natural identifications of their respective sources and targets.
\end{cor}
\begin{proof}
    The proof of \cite[Proposition 2.6]{Sheridan2013} applies verbatim; the only foundational inputs it uses are Theorem \ref{thm:wpcy} (to show that the bottom horizontal map is an isomorphism), Theorem \ref{thm:oc_int}, and the standard fact that the pairing $\langle -,-\rangle$ makes $QH^*(X;\Lambda^\sm)$ into a Frobenius algebra (to show that the diagram commutes).
\end{proof}

Theorems \ref{thm:mukai_int} and \ref{thm:wpcy} allow us to implement the following result of Sanda and Ganatra \cite{Sanda,Ganatra2016} in our setting (we state the criterion in the Calabi--Yau case, where it is cleanest, but see the references for natural generalizations):

\begin{thm}[Theorem 1.1 of \cite{Sanda}, Theorems 1 and 8 of \cite{Ganatra2016}]\label{thm:aut_gen}
    Suppose that $X$ is connected, and the grading group is $\G = \Z$. 
    If $\cA \subset \fuk^\sm(X,D;\Lambda^\sm)^\bc$ is homologically smooth, then:
    \begin{itemize}
        \item $\cA$ split-generates $\fuk^\sm(X,D;\Lambda^\sm)^\bc$;
        \item the following four maps are isomorphisms:
        $$QH^*(X;\Lambda^\sm) \xrightarrow{\cC\cO_\Lambda^\sm} HH^*(\fuk^\sm(X,D;\Lambda^\sm)^\bc) \to HH^*(\cA)$$
        and
        $$HH_*(\cA)[-n] \to HH_*(\fuk^\sm(X,D;\Lambda^\sm)^\bc)[-n] \xrightarrow{\cO\cC_\Lambda^\sm} QH^*(X;\Lambda^\sm).$$
    \end{itemize}
\end{thm}
\begin{proof}
    A brief examination of the proof of the first point given in \cite[Section 5]{Ganatra2016} shows that the only foundational inputs needed to prove the first point are Abouzaid's split-generation criterion \cite{Abouzaid2010a} (which holds in this case by \cite[Theorem 1.13]{relfukii}), Theorem \ref{thm:mukai_int}, and Corollary \ref{cor:co_oc_dual}. 
    Otherwise, the proof applies verbatim.

    A similarly brief examination shows that the proof that the second point is a consequence of the first relies only on the foundational results of Theorem \ref{thm:oc_int} (to show $\cO\cC_\Lambda$ is surjective), Theorem \ref{thm:mukai_int} (to show it is injective), and Corollary \ref{cor:co_oc_dual} (to conclude that $\cC\cO_\Lambda$ is also an isomorphism), and otherwise applies verbatim.
\end{proof}

We now turn to results about cyclic homology. 
We state all our results for negative cyclic homology, $HC_*^-$, but the natural analogues of all these results hold for the positive and periodic versions.

\begin{thm}\label{thm:OC_minus_bc}
    There is a homomorphism of $\Lambda^\big[[u]]$-modules extending the big open--closed map, called the \emph{big negative cyclic open--closed map}
    $$\cO\cC_\Lambda^{-,\big}: fHC_*^-(\fuk^\big(X,D;\Lambda^\big)^{sbc})[-n] \to QH^{*}(X;\Lambda^\big)[[u]],$$
    where $fHC_*^-$ denotes the filtered negative cyclic homology. 
    Quotienting by $\fm^\bulk_\Lambda$, we obtain a homomorphism of $\Lambda^\sm[[u]]$-modules extending the small open--closed map, called the \emph{small negative cyclic open--closed map}
    $$\cO\cC_\Lambda^{-,\sm}: HC_*^-(\fuk^\sm(X,D;\Lambda^\sm)^\bc) [-n] \to QH^*(X;\Lambda^\sm)[[u]].$$
\end{thm}

Theorem \ref{thm:OC_minus_bc} will be proved in Section \ref{sec:cyc oc}. It builds on the construction of the cyclic open--closed map in \cite{Ganatra:cyclic}, in the positively monotone or exact case. 
We have:

\begin{lem}\label{lem:OCiso}
    If $\cO\cC_\Lambda^\sm$ is an isomorphism, then so are $\cO\cC_\Lambda^\big$, $\cO\cC_\Lambda^{-,\sm}$, and $\cO\cC_\Lambda^{-, \big}$.
\end{lem}
\begin{proof}
    If $\cO\cC_\Lambda^{\sm}$ is an isomorphism, then so is $\cO\cC_\Lambda^\big$, by a spectral sequence argument for the $\fm^\bulk_\Lambda$-adic filtration; and so is $\cO\cC_\Lambda^{-,\sm}$, by a spectral sequence argument for the $u$-adic filtration. 
    Finally, if $\cO\cC_\Lambda^{\big}$ is an isomorphism, then so is $\cO\cC_\Lambda^{-,\big}$ by a spectral sequence argument for the $u$-adic filtration.
\end{proof}

\begin{thm}\label{thm:oc_hres_bc}
    The big negative cyclic open--closed map respects pairings:
    $$\langle \cO\cC_\Lambda^{-,\big}(\alpha),\cO\cC_\Lambda^{-,\big}(\beta)\rangle = (-1)^\signn \langle \alpha,\beta \rangle_{res},$$
    where $\langle -,-\rangle$ denotes the $u$-sesquilinear extension of the integration pairing $\langle-,-\rangle$ on $QH^*(X;\Lambda^\big)[[u]]$, and $\langle -,-\rangle_{res}$ denotes the higher residue pairing on $fHC_*^-(\fuk^\big(X,D;\Lambda^\big)^{sbc})$, with conventions as in \cite[Definition 5.33]{Sheridan_formulae}. 
    Quotienting by $\fm^\bulk_\Lambda$, we obtain that the small negative cyclic open--closed map respects pairings.
\end{thm}

Theorem \ref{thm:oc_hres_bc} will be proved in Section \ref{sec:oc pair}.

Our next result will say that $\cO\cC_\Lambda^{-,\big}$ respects connections. 
In order to formulate the result, it will be convenient to choose a homogeneous basis $(\beta^i)$ of $H^*(X;\C)$, and let $(\beta_i)$ denote the dual basis with respect to the intersection pairing; and let $\novblam_i$ denote the corresponding bulk generators of $\Lambda^\big$. We define a $\Lambda^\big$-module 
$$\Omega^\big_\Lambda := \Lambda^\big \langle d\log T, d\novblam_i \rangle,$$
together with the obvious derivation
$$D^\big_\Lambda:\Lambda^\big \to \Omega^\big_\Lambda.$$

The ($u$-)connection on quantum cohomology is the Dubrovin/Givental connection:
\begin{align}
    u\nabla^{DG}_\Lambda: QH^*(X;\Lambda^\big)[[u]] & \to \Omega^\big_\Lambda \otimes QH^*(X;\Lambda^\big)[[u]],\\
    u\nabla^{DG}_\Lambda(\alpha) &:= uD^\big_\Lambda(\alpha) - d\log T \otimes [\omega] \star \alpha - \sum_i d\novblam_i \otimes \gamma^i \star \alpha.
\end{align}
Strictly speaking, there is no well-defined map $\nabla^{DG}_\Lambda$, as we can't multiply this expression by $u^{-1}$; rather, $u\nabla^{DG}$ is a well-defined map which has all the properties expected of a connection multiplied by $u$. 
We call such an object a \emph{$u$-connection}.

The $u$-connection on cyclic homology is given by the Getzler--Gauss--Manin connection \cite{Getzler_GM}, with conventions as in \cite{Sheridan_formulae}:
$$u\nabla^{GGM} : fHC_*^-(\fuk^\big(X,D;\Lambda^\big)^\sbc) \to \Omega^\big_\Lambda \otimes fHC_*^-(\fuk^\big(X,D;\Lambda^\big)^\sbc).$$

\begin{thm}\label{thm:oc_conn_bc}
    The big negative cyclic open--closed map respects $u$-connections:
    $$\left(\id \otimes \cO\cC_\Lambda^{-,\big} \right) \circ u\nabla^{GGM} = u\nabla^{DG}_\Lambda \circ \cO\cC_\Lambda^{-,\big}.$$
\end{thm}

Theorem \ref{thm:oc_conn_bc} will be proved in Section \ref{sec:oc conn}. A similar result has been proved, under strong regularity hypotheses on moduli spaces of holomorphic curves, in \cite[Theorem 1.7]{Hugtenburg:OC}.

\subsection{Implications for mirror symmetry of compact Calabi--Yau varieties}\label{sec:frob}

Let us expand on the implications of our results for mirror symmetry, in light of \cite{Ganatra2015}. 
For concreteness, we consider a `Greene--Plesser mirror pair', as defined in \cite{sheridan2021homological}. 
On the $A$-side of mirror symmetry, the construction gives rise to a Calabi--Yau variety $X$ with a simple normal-crossings divisor $D$, whose irreducible components $(D_p)_{p \in P}$ are indexed by some set $P$. 
The coefficient ring of the small relative Fukaya category $\fuk^\sm(X,D)$ is $R^\sm$; and a K\"ahler class $\kappa$ determines the homomorphism $d(\kappa)^*:R^\sm \to \Lambda^\sm_{\ge 0}$. 
We will assume that $\kappa$ is integral, so that the homomorphism factors through $\C[[T]] \subset \Lambda^\sm_{\ge 0}$. 

On the $B$-side, the construction gives rise to a family of Calabi--Yau varieties $\check{\mathcal{X}}$ over $\mathbb{A}^P$. 
By \cite[Theorem C]{Ganatra2023integrality}, there exist $\psi_p \in R^\sm$ such that there is a quasi-equivalence of $\C((T))$-linear $A_\infty$ categories
$$D^\pi\fuk^\sm(X,D;\Lambda^\sm)^\bc \simeq D^bCoh(\check{X}_b),$$
where $b \in \mathbb{A}^P_{\C((T))}$ is defined by $b_p = d(\kappa)^*(\nov_p \cdot \psi_p)$. 

By \cite[Theorem A]{Ganatra2015} (which was originally contingent on the results now proved in this paper, as well as analogous results on the $B$-side, since proved in \cite{Tu:HKR}, where the author describes it as a `folklore' result), it follows that Hodge-theoretic mirror symmetry holds: there is an isomorphism of $\Z$-graded VSHS over $\C((T))$,
\begin{equation}\label{eq:HodgeMS}
    \mathcal{H}_A^\sm(X) \cong \mathcal{H}_B^\sm(\check{X}).
\end{equation}
We recall the relevant definitions from \cite{Ganatra2015}. 
Recall that a $\Z$-graded VSHS over $\C((T))$ is equivalent to a $\Z/2$-graded vector space over $\C((T))$ equipped with a filtration, a connection in the $T$-direction which satisfies Griffiths transversality with respect to the filtration, and a covariantly constant nondegenerate pairing (see \cite[Lemma 2.7]{Ganatra2015}).  
The small $A$-VSHS $\mathcal{H}_A^\sm(X)$ is defined to be $H^*(X;\C((T)))[n]$, equipped with the degree filtration, Dubrovin--Givental connection, and integration pairing (see \cite[Definition 3.1]{Ganatra2015}). 
The small $B$-VSHS $\mathcal{H}_B^\sm(\check{X})$ is defined to be $H^*_{dR}(\check{X})$, equipped with the Hodge filtration, Gauss--Manin connection, and twisted integration pairing (see \cite[Definition 3.6]{Ganatra2015}). 

It is explained in \cite[Appendix C]{Sheridan2017} how to use this result, for different integral K\"ahler forms $\kappa$, to characterize the mirror map $(\psi_p)$. 
In particular, we may write down formulae for the mirror map in terms of the solutions to hypergeometric differential equations; this is used in the proof of \cite[Theorem B]{Ganatra2023integrality}. 

It is reviewed in \cite{Ganatra2015} how Hodge-theoretic mirror symmetry \eqref{eq:HodgeMS} implies the classical enumerative predictions of mirror symmetry for, e.g., the quintic threefold. 
In particular, with the foundations provided by this paper to justify the application of \cite[Theorem A]{Ganatra2015} and the proof of homological mirror symmetry given in \cite{sheridan2015homological}, we obtain a new proof of the enumerative predictions of Candelas--de la Ossa--Green--Parkes \cite{CdlOGP}, independent of those given by Givental \cite{Givental:EGW} or Lian--Liu--Yau \cite{LLY,LLY2,LLY3,LLY4} (where the definition of the Gromov--Witten invariants involved is the symplectic one, by \cite[Lemma 5.4]{relfukii}, which is equivalent to the algebraic definition by \cite{LiTian}).

In fact, whereas \cite{Ganatra2015} restricted its attention to the case of small quantum cohomology (which of course suffices to completely determine the genus-zero Gromov--Witten invariants of a Calabi--Yau threefold), the results in this paper are sufficiently general to prove the analogous statements for big quantum cohomology, as alluded to in \cite[Remark 1.20]{Ganatra2015}. 
Namely, we have:

\begin{thm}\label{thm:vshs}
    Suppose that $\fuk^\sm(X,D;\Lambda^\sm)^\bc$ has a homologically smooth subcategory $\cA$, the grading group is $\G = \Z$, and $X$ is connected. 
    Then the big cyclic open--closed map
    $$\cO\cC^{-,\big}:fHC_*^-(\fuk^\big(X,D;\Lambda^\big)^\sbc)[-n] \to QH^*(X;\Lambda^\big)[[u]]$$
    is an isomorphism of polarized VSHS over $\Lambda^\big$. 
\end{thm}
\begin{proof}
    By Theorem \ref{thm:aut_gen}, $\cO\cC^\sm_\Lambda$ is an isomorphism; hence so is $\cO\cC^{-,\big}_\Lambda$, by Lemma \ref{lem:OCiso}. 
    It respects pairings by Theorem \ref{thm:oc_hres_bc}, and connections by Theorem \ref{thm:oc_conn_bc}, hence is an isomorphism of polarized VSHS.
\end{proof}

We now explain the sense in which Theorem \ref{thm:vshs} means that the genus-zero Gromov--Witten invariants of $X$ ``can be extracted from $\fuk^\sm(X,D;\Lambda^\sm)^\bc$'', under the given hypothesis, together with the assumption that the cohomology of $X$ satisfies the Hard Lefschetz property. 

\begin{rem}
    Note that the hypotheses that $\fuk^\sm(X,D;\Lambda^\sm)^\bc$ admits a homologically smooth subcategory, the Fukaya category is $\Z$-graded, $X$ is connected, and its cohomology satisfies the Hard Lefschetz property, are verified for generalized Greene--Plesser mirror pairs in \cite[Theorem D]{sheridan2021homological},\footnote{More precisely, our constructions in this paper only apply when the `MPCS' condition (guaranteeing that $X$ is smooth) and `embeddedness' condition (guaranteeing that the Lagrangians constructed are embedded -- this holds for all Greene--Plesser mirror pairs) hold in that paper.} and for a class of Batyrev mirror pairs in \cite[Theorem B]{GHHPS}. 
    In particular, the following arguments apply in these cases.
\end{rem}

First, we observe that the hypothesis implies that $\cC\cO^\sm$ is an isomorphism, by Theorem \ref{thm:aut_gen}, and hence that $\fuk^\big(X,D;\Lambda^\big)^\sbc$ is a versal deformation of $\fuk^\sm(X,D;\Lambda^\sm)^\bc$, by Corollary \ref{cor:vers}. 
The versal deformation of $\fuk^\sm(X,D;\Lambda^\sm)^\bc$ is unique up to a formal change of variables and a quasi-equivalence, by a standard deformation theory argument such as \cite[Theorem 3.3]{Sheridan_CDM}. 

We now recall that, under the assumption that the cohomology of $X$ has the Hard Lefschetz property, the natural splitting
$$\sigma^\sm: QH^*(X;\Lambda^\sm) \to QH^*(X;\Lambda^\sm)[[u]]$$
(given by the inclusion of power series which are constant in $u$) of the projection map $QH^*(X;\Lambda^\sm)[[u]] \to QH^*(X;\Lambda^\sm)$ may be characterized in terms of the monodromy weight filtration of the small Dubrovin--Givental connection, as reviewed in \cite[Sections 2 and 3]{Ganatra2015}. 
A splitting $\sigma^\sm$ is equivalent to a choice of `opposite subspace' for our small VSHS $QH^*(X;\Lambda^\sm)[[u]]$, in the sense of Barannikov \cite{Barannikov} (cf. \cite[Definition 2.10]{CIT}). 
Explicitly, the opposite subspace is the $\Lambda^\sm[u^{-1}]$-submodule of $QH^*(X;\Lambda^\sm)((u))$ spanned by $u^{-1}\im(\sigma)$.

The splitting/opposite subspace can be uniquely extended to one for the big VSHS, in such a way that the opposite subspace is invariant under the Dubrovin--Givental connection. 
The resulting splitting 
$$\sigma^\big: QH^*(X;\Lambda^\big) \to QH^*(X;\Lambda^\big)[[u]]$$
is easily shown to be the natural one, corresponding to power series which are constant in $u$, by arguing order-by-order in the $\fm^\bulk_\Lambda$-adic filtration. 
It is straightforward to observe that this opposite subspace is isotropic for the pairing $\langle-,-\rangle$, and furthermore is graded.

We now observe that the big VSHS is miniversal, in the sense of \cite{Barannikov} (cf. \cite[Definition 2.8]{CIT}); explicitly, if we take $s_0$ to be the section corresponding to the identity $1 \in QH^*(X;\Lambda^\big)[[u]]$, then the map
\begin{align*}
    (\fm^\bulk_\Lambda/(\fm^\bulk_\Lambda)^2)^\vee & \to QH^*(X;\Lambda^\big)[[u]]/uQH^*(X;\Lambda^\big)[[u]] = QH^*(X;\Lambda^\big) \\
    v & \mapsto (\nabla^{DG}_\Lambda)_v (s_0)
\end{align*}
is identified with the map $QH^*(X;\Lambda^\big) \to QH^*(X;\Lambda^\big)$ given by quantum cup product with the identity, which is of course an isomorphism. 
Furthermore, the section $s_0$ has degree $0$, so it defines a \emph{dilaton shift} in the sense of \cite[Section 2.2.2]{CIT}.

As the VSHS is miniversal and comes equipped with an opposite subspace which is isotropic and graded, and a dilaton shift, it determines a formal Frobenius manifold structure on the formal neighbourhood of the origin in $QH^*(X;\Lambda^\sm)$, by \cite[Proposition 2.11]{CIT}; as the opposite subspace is the natural one, this is the natural Frobenius manifold structure associated to the genus-zero Gromov--Witten invariants of $X$, defined for example in \cite{Manin}. 
Note that, while the base of the versal deformation space does not \emph{a priori} come equipped with a natural parametrization, the `flat coordinates' defined in \cite[Section 2.2.2]{CIT} remove this ambiguity, and define a natural parametrization of the base (cf. \cite[Section 2.5]{Ganatra2015}). 

We have explained how our results allow to extract the Frobenius manifold structure on quantum cohomology from $\fuk^\sm(X,D;\Lambda^\sm)^\bc$, under hypotheses which have been verified for a large class of generalized Greene--Plesser and Batyrev mirror pairs in \cite{sheridan2021homological,GHHPS}. 
As the homological mirror symmetry results proved in those references identify (the triangulated envelope of) $\fuk^\sm(X,D;\Lambda^\sm)^\bc$ with a mirror category of coherent sheaves $D^bCoh(\check{X})$, it follows that the Frobenius manifold structure on quantum cohomology can be extracted from $D^bCoh(\check{X})$ by applying the same procedure: we pass to the versal deformation of $D^bCoh(\check{X})$ (we note that it may be curved), take the VSHS defined by its filtered cyclic homology, take the opposite subspace defined by the monodromy weight filtration and extend it to the versal deformation, observe that it is isotropic, graded, and admits a dilaton shift (because the $A$-model VSHS does -- note that these are properties rather than additional structures, with the exception of the dilaton shift which is defined up to a complex scalar, which can however be fixed up to a sign by considering the pairing $\langle s_0,\nabla_{T \partial_T}^n s_0 \rangle$, cf. \cite[Section 6.2]{Ganatra2015}), then construct the associated Frobenius manifold. 

In order to compare the resulting Frobenius manifold with the $B$-model Frobenius manifold associated to $\check{X}$ in \cite{BK}, and hence prove the version of closed-string mirror symmetry proposed there and in \cite{Barannikov,Barannikov:MS}, it remains to show that the VSHS defined on the cyclic homology of the versal deformation of $D^bCoh(\check{X})$ is isomorphic to the one constructed in \cite[Section 4]{Barannikov}, extending the result proved in \cite[Theorem 0.1]{Tu:HKR}.

\paragraph{Acknowledgments:} We are grateful to Tim Perutz, our collaborator on \cite{Ganatra2015}. S.G. was supported by NSF grant CAREER DMS-2048055 and a Simons Fellowship (award number 1031288).
N.S. was supported by an ERC Starting Grant (award number 850713 – HMS), a Royal Society University Research Fellowship, the Leverhulme Trust through the Leverhulme Prize, and a Simons Investigator award (award number 929034).

\section{Algebraic preliminaries}

We adopt the notation and terminology wholesale from \cite[Section 2]{relfukii}. 
We particularly recall the device used to keep track of Koszul signs for homomorphisms of non-zero degrees: given graded modules $M$ and $N$, we regard a homomorphism $\alpha \in Hom^{|\alpha|}(M,N)$ of degree $|\alpha|$ as being equivalent to a grading-preserving homomorphism $Hom(\sigma(\alpha) M, N)$, where $\sigma(\alpha)$ is a graded $\Z_2$-torsor in degree $|\alpha|$ and we use the shorthand $\sigma(\alpha) M:= \sigma(\alpha) \otimes M$. 

For example, the open--closed maps of all flavours have degree $n$: so we define $\sigma(\cO\cC) = \sigma(n)$ to be the graded $\Z_2$-torsor in degree $n$, and the cyclic open--closed map as a homomorphism
$$\cO\cC^{-,\big}:\sigma(\cO\cC)fHC_*^-(\fuk^\big(X,D;\Lambda^\big)^\sbc) \to QH^*(X;\Lambda^\big)[[u]].$$

\subsection{Small bounding cochains}\label{sec:sbc}

In this section, we establish Definition--Lemma \ref{deflem:bigsmall}. 
We start by defining
$$\fuk^\big(X,D;\Lambda^\big) := \fuk^\big(X,D) \otimes_{R^\big} \Lambda^\big_{\ge 0},$$
where the morphism $R^\big \to \Lambda^\big_{\ge 0}$ is defined in \eqref{eq:RbigLambig}.

The objects of $\fuk^\big(X,D;\Lambda^\big)^\sbc$ are objects $(L,b)$ of $\fuk^\sm(X,D;\Lambda^\sm)^\bc$: a Lagrangian brane $L$ equipped with a (`small') bounding cochain $b$. 
Via the natural homomorphism $\Lambda^\sm_{\ge 0} \to \Lambda^\big_{\ge 0}$, such $b$ can be regarded as an endomorphism of the same object $L$, now regarded as an object of $\fuk^\big(X,D;\Lambda^\big_{\ge 0})$. 
In particular, we may regard $(L,b)$ as a `pre-bounding cochain' in the sense of \cite[Definition 2.13]{relfukii}, i.e.  an object of $(\fuk^\big(X,D; \Lambda^\big_{\ge 0})^\prebc$. 
This object may have nonvanishing curvature, however, its curvature will lie in $\fm^\bulk_\Lambda$, because $b$ has vanishing curvature in $\fuk^\sm(X,D) \otimes \Lambda^\sm_{\ge 0}$. 
It follows that the full subcategory of $(\fuk^\big(X,D; \Lambda^\big_{\ge 0})^\prebc$ consisting of such objects is linear over the coefficient ring $\Lambda^\big_{\ge 0}$, equipped with the $\fm^\bulk_\Lambda$-adic filtration. 
We now base-change it to $\Lambda^\big$, and define $\fuk^\big(X,D;\Lambda^\big)^\sbc$ to be this category: it is evident that it has the desired properties.

\subsection{Weak proper Calabi--Yau structures}

Let $\cC$ be an $R$-linear $A_\infty$ category, in the sense of \cite[Definition 2.10]{relfukii}, $\bimod{\cC}{\cC}$ the $R$-linear unital differential graded category of $(\cC,\cC)$-bimodules, $\cC_\Delta$ the diagonal bimodule. 
Recall the definition of the \emph{linear dual} $\mathcal{M}^\vee$ of a bimodule $\mathcal{M}$ given in \cite[Equation (2.11)]{Seidel_II}.

\begin{defn}[Cf. Section A.5 of \cite{Sheridan2013}]
    An \emph{$n$-dimensional weak proper Calabi--Yau structure} ($n$-wpCY structure) on $\cC$ is an isomorphism $\sigma(n)\cC_\Delta \cong \cC_\Delta^\vee$ in $H(\bimod{\cC}{\cC})$.
\end{defn}

\begin{lem}\label{lem:wpcy_bc}
    An $n$-wpCY structure on $\cC$ induces one on $\cC^\bc$.
\end{lem}
\begin{proof}
    Follows from the fact that the pullback functor 
    $$(F \otimes F)^*: \bimod{\cC}{\cC} \to \bimod{\cC^\bc}{\cC^\bc},$$
    which is induced by the functor $F:\cC^\bc \to \cC$ from \cite[Lemma 2.14]{relfukii}, sends $\cC_\Delta \mapsto \cC^\bc_\Delta$, respects linear duality, and also sends identity morphisms to identity morphisms.
\end{proof}

We also recall the tensor product $\cM \otimes_\cC \cN$ of $(\cC,\cC)$-bimodules $\cM$ and $\cN$ from \cite[Equations (2.12-13)]{Seidel_II}. 
There is a natural isomorphism of chain complexes
\begin{equation}\label{eq:lin_duality}
hom_{\bimod{\cC}{\cC}}(\cC_\Delta,\cM^\vee) \cong fCC_*(\cC_\Delta \otimes_\cC \cM)^\vee
\end{equation}
arising from hom-tensor adjunction (see \cite[Equation (2.26)]{Seidel_II}).

Recall that there is a closed bimodule homomorphism
$$\cC_\Delta \otimes_\cC \cM \to \cM$$
defined in \cite[Equation (2.7)]{Seidel_nat_transf} (it will not be relevant for us whether this homomorphism is a quasi-isomorphism in our setting).
This induces a map
$$fCC_*(\cC)^\vee \to fCC_*(\cC_\Delta \otimes_\cC \cC_\Delta)^\vee.$$
Composing this with \eqref{eq:lin_duality}, for $\cM = \cC_\Delta$, we obtain a chain map
$$G_\cC: fCC_*(\cC)^\vee \to hom_{\bimod{\cC}{\cC}}(\cC_\Delta,\cC_\Delta^\vee).$$

\begin{lem}\label{lem:wpcy_cd}
    The diagram
    $$\begin{tikzcd}
        CC_*(\cC^\bc)^\vee \ar[r,"G_{\cC^\bc}"]  & hom_{\bimod{\cC^\bc}{\cC^\bc}}(\cC^\bc_\Delta,(\cC^\bc)_\Delta^\vee) \\
        fCC_*(\cC)^\vee \ar[r,"G_\cC"]\ar[u,"(F_*)^\vee"] & hom_{\bimod{\cC}{\cC}}(\cC_\Delta,\cC_\Delta^\vee) \ar[u,"(F \otimes F)^*"] 
    \end{tikzcd}$$
    commutes.
\end{lem}
\begin{proof}
    Follows from the definitions.
\end{proof}

\begin{cor}\label{lem:wpcy_gc}
    Suppose that we have $\phi \in fCC_*(\cC)^\vee$, such that $G_\cC(\phi)$ is an $n$-wpCY structure on $\cC$. 
    Then $(F_*)^\vee \phi \in CC_*(\cC^\bc)^\vee$ induces the corresponding $n$-wpCY structure on $\cC^\bc$.
\end{cor}
\begin{proof}
    Follows from the definitions.
\end{proof}

\subsection{Cyclic homology}\label{sec:cyc}

Let $\cC$ be an $R$-linear $A_\infty$ category. 
We define a new such category $\cC^+$, by 
$$ hom_{\cC^+}(C_0,C_1) := \begin{cases}
                            hom_\cC(C_0,C_1) & \text{if $C_0 \neq C_1$}\\
                            hom_\cC(C_0,C_1) \oplus R \cdot e^+ & \text{if $C_0=C_1$,}
\end{cases}  
$$
with $e^+$ in degree $0$. 
We define $\mu^s(\ldots, e^+,\ldots) = 0$ for all $s \neq 2$, and 
$$\mu^2( e^+,a) = a = -\mu^2(a, e^+),$$
where $\sigma(\mu) \sigma(e^+) = \sigma(0)$ (here, $\sigma(e^+)$ is the $\sigma(1)$ out the front of $e^+$ in $\cC^+(C,C) = \sigma(1)^\vee hom_{\cC^+}(C,C)$), and leaving all other structure maps $\mu^*$ unchanged.
Then $\cC^+$ is a strictly unital filtered $A_\infty$ category, with strict units $e^+$. 

We define the subcomplex $D_* \subset CC_*(\cC^+)$ of \emph{degenerate elements}, generated by $c_0[\ldots|c_s]$ such that $c_i = \sigma e^+$ for some $i>0$, together with the length-zero chains $\sigma e^+$. 
We define the \emph{non-unital unfiltered Hochschild chains}, $CC_*^{nu}(\cC) := CC_*(\cC^+)/D_*$, and the non-unital filtered Hochschild chains $fCC^{nu}_*(\cC) = \overline{CC}_*^{nu}(\cC)$. 
These decompose as a direct sum:
\begin{align*}
    fCC^{nu}_*(\cC) &\cong fCC^\vee_*(\cC) \oplus fCC^\wedge_*(\cC),\qquad\text{where}\\
    fCC^{\vee}_*(\cC) &= fCC_*(\cC),\\
    fCC^{\wedge}_*(\cC) &= \sigma(e^+)^\vee fCC_*(\cC).
\end{align*}
Here $fCC^{\vee}_*$ is interpreted as those Hochschild chains which do not have an entry equal to an $e^+$, while $fCC^{\wedge}_*$ is interpreted as those which have the first entry equal to an $e^+$. 
The Hochschild differential on this complex is equal to the usual Hochschild differential $b$ on the factor $fCC^{\vee}_*$, while it equals
$$b(e^+[\alpha]) = e^+[b_{\wedge\wedge}(\alpha)] + b_{\wedge \vee}(\alpha)$$
on $fCC^{\wedge}_*$, where
\begin{align}
\label{eq:b hat}
    b_{\wedge \wedge}(a_0[a_1|\ldots|a_s]) &=  \sum_{i,j} e^+[a_1|\ldots|\mu^{j-i}(a_{i+1},\ldots,a_j)|\ldots|a_s],\\
    b_{\wedge \vee}(a_0[a_1|\ldots|a_s]) &= a_0[a_1|\ldots|a_s] - a_s[a_0|\ldots|a_{s-1}],
\end{align}
where as usual $\sigma(b) = \sigma(\mu)$.

When $\cC$ is an ordinary $A_\infty$ category and c-unital, the composition of the natural maps
\[ CC_*(\cC) \hookrightarrow CC_*(\cC^+) \to CC_*^{nu}(\cC)\]
is a quasi-isomorphism (compare \cite[Section 1.4]{Loday_CH}). 
It follows that if $\cC$ is filtered and admits an HH-unit (in the sense of \cite[Definition 2.17]{relfukii}), the analogue holds for $fCC_*(\cC) \to fCC_*^{nu}(\cC)$, by a spectral sequence argument for the natural filtration. 

Now define the Connes differential $B: fCC_*^{nu}(\cC) \to fCC_*^{nu}(\cC)$ by
\begin{equation}
\label{eqn:ConnesB}
 B(c_0[\ldots|c_s]) := \sum_j e^+[c_{j+1}|\ldots|c_s|c_0|\ldots|c_j]
\end{equation}
where $\sigma(B)\sigma(CC_*) = \sigma(CC_*)\sigma(e^+)$.    
It has degree $-1$, and satisfies $B^2 = 0$ and $bB+Bb = 0$. 
We note that it vanishes on $fCC^{\wedge}_*(\cC)$, and sends $fCC^{\vee}_*(\cC) \mapsto fCC^{\wedge}_*(\cC)$. 

We define $fCC_*^-(\cC) := fCC_*^{nu}(\cC)[[u]]$. 
Here the notation `$[[u]]$' means we tensor with $R[u]$ then take the completion with respect to the $u$-adic filtration. 

We equip $fCC_*^-(\cC)$ with the differential $b+uB$, and denote its cohomology by $fHC_*^-(\cC)$.  
If the coefficient ring $R = (\G,R,\cF_{\ge \bullet},d)$ has $\cF_{\ge 1} = 0$, then $\cC$ is uncurved, and $CC_*^- = fCC_*^-$, and we write $HC_*^-$ instead of $fHC_*^-$.

Filtered cyclic homology is functorial in $\cC$, in the sense that an $A_\infty$ functor $F:\cC \to \cD$ induces a map $F_*:fHC_*^-(\cC) \to fHC_*^-(\cD)$; this follows from \cite[Lemma 3.26]{Sheridan_formulae}, adapted to the filtered setting. 

\begin{lem}\label{lem:HC_C_Cbc}
    Let $F:\cC^{\prebc} \to \cC$ be the $A_\infty$ functor from \cite[Lemma 2.14]{relfukii}. 
    It induces a map
    $$F_*: fHC_*^-(\cC^{\prebc}) \to fHC_*^-(\cC),$$
    and its restriction to the subcategory $\cC^\bc \subset \cC$ induces a map
    $$F_*: HC_*^-(\cC^\bc) \to fHC_*^-(\cC).$$
\end{lem}
\begin{proof}
    Follows from \cite[Lemma 3.26]{Sheridan_formulae}, adapted to the filtered setting. 
\end{proof}

\subsection{Pairings}

For any finite-rank free $R$-module $M$, we can define the supertrace
$$\mathsf{str}: Hom(M,M) \to R$$
to be the composition
$$Hom(M,M) \cong Hom(M,R) \otimes M \xrightarrow{\mathrm{ev}} R.$$

Let $\cC$ be an $R$-linear $A_\infty$ category whose morphism spaces are finite-rank free $R$-modules (i.e., finite-rank and free on the cochain level, not just on the cohomology level). 
We define the \emph{Mukai pairing}  
$$\langle -,-\rangle_{Muk}: fHH_*(\cC) \otimes fHH_*(\cC) \to \BbK,$$
and its $u$-sesquilinear extension, the \emph{higher residue pairing}
$$\langle - ,-\rangle_{res}: fHC_*^-(\cC) \otimes fHC_*^-(\cC) \to \BbK[[u]],$$
to both be induced by the following chain map: if $\alpha = c_0[c_1|\ldots|c_s]$ and $\beta = d_0[d_1|\ldots|d_t]$ are generators of $CC_*(\cC)$ (respectively $CC_*^{nu}(\cC)$), then
\begin{multline}
\label{eqn:mukform}
 \langle \alpha,\beta \rangle_{Muk} := \sum_{j,k,\ell,m}\mathsf{str} (c \mapsto \\
 \left. \mu_1^*(c_{k+1},\ldots,c_s,c_0,\ldots,c_j,\mu_2^*(c_{j+1},\ldots,c_k,c,d_{m+1},\ldots,d_s,d_0,\ldots,d_\ell),\ldots,d_m) \right),
 \end{multline}
 where $\sigma(\alpha) \sigma(\beta) = \sigma(\mu_1)\sigma(\mu_2)$. Here, $\sigma(\alpha)$ is the $\sigma(-1)$ out the front of the copy of $CC_*(\cC)$ containing $\alpha$, and similarly for $\sigma(\beta)$; and $\mu_1 = \mu_2 = \mu_\cC$, but we use different notation to distinguish the copies of $\sigma(\mu_i)$ associated with them. 
If the expression in \eqref{eqn:mukform} is not composable in $\cC$, we set the summand to be $0$. 
Here `$c$' represents an element in the corresponding space $\cC(C_0,C_1)$. 

Note that if $\alpha = (\alpha^\vee,\alpha^\wedge) \in fCC^{\vee}_* \oplus fCC^{\wedge}_*$, and similarly for $\beta$, then 
\begin{equation}
    \langle \alpha,\beta \rangle_{Muk} = \langle \alpha^\vee,\beta^\vee \rangle_{Muk,\vee\vee} + \langle \alpha^\wedge,\beta^\vee \rangle_{Muk,\wedge\vee} + \langle \alpha^\vee,\beta^\vee \rangle_{Muk,\vee\wedge}(\alpha^\vee,\beta^\wedge)
\end{equation}
where $\langle,\rangle_{Muk,\vee\vee} = \langle,\rangle_{Muk}$ is the usual Mukai pairing on $fCC^{\vee}_*$; and if $\alpha = c_0[c_1|\ldots|c_s]$ and $\beta = d_0[d_1|\ldots|d_t]$ a generators of $fCC_*$, then
\begin{align*} 
\langle \alpha,\beta \rangle_{Muk,\wedge\vee} &= \sum_{m} \mathsf{str}(c \mapsto \mu^*(c_0,\ldots,c_s,c,d_{m+1},\ldots,d_0,\ldots,d_m)),\\
\langle \alpha,\beta \rangle_{Muk,\vee\wedge} &= - \sum_k \mathsf{str}(c \mapsto \mu^*(c_{k+1},\ldots,c_0,\ldots,c_k,c,d_0,\ldots,d_t)).\\
\end{align*}

\begin{lem}\label{lem:C_Cbc_pairings}
The natural maps $fHH_*(\cC^{\prebc}) \to fHH_*(\cC)$ and $fHC_*^-(\cC^{prebc}) \to fHC_*^-(\cC)$ from Lemma \ref{lem:HC_C_Cbc}, respect pairings.
\end{lem}
\begin{proof}
    The results hold on the chain level, by inspection of the formulae.
\end{proof}

\subsection{Connections}

Let $R = (\G,R,\cF_{\ge \bullet},d)$ be a coefficient ring, $\Omega_R$ a free $R$-module of finite rank, and $D:R \to \Omega_R$ a derivation. 
Explicitly, in our context this means that $d_\Omega \circ D = D \circ d_R$, and $D(fg) = D(f) g + fD(g).$ 

\begin{eg}\label{eg:Omega R big}
    Let $R = R^\big$ be the coefficient ring of the big relative Fukaya category. 
    For each $q \in Q$, define $\Omega^q_R := R\cdot d\log \nov_q$, and $D^q_R: R \to \Omega^q_R$ by
    \begin{align*}
        D^q_R(\nov^u) &= \nov^u \cdot \novb^k \otimes (u \cdot V_q) d\log \nov_q \qquad \text{for $u \in \NE$};\\
        D^q_R\left(\frac{\novb_i^{k}}{k!}\right) &= 0,
    \end{align*}
    and extending to all of $R^\big$ by the Leibniz rule and continuity.
    For each basis element $\beta_i$ of $CM^{2-*}(f,g)$, define $\Omega^i_R := R \cdot d\novb_i$, and $D^i_R: R \to \Omega^i_R$ by
    \begin{align*}
        D^i_R(\nov^u) &= 0 \qquad \text{for $u \in \NE$};\\
        D^i_R\left(\frac{\novb_j^{k}}{k!}\right) &= \begin{cases}
                                                        \frac{\novb_i^{k-1}}{(k-1)!}\cdot d\novb_i & \text{ if $i=j$;}\\
                                                        0 & \text{otherwise.}
        \end{cases}                 
    \end{align*}
    Now define a derivation $D_R^\big:R \to \Omega^\big_R$ by
    \begin{align*}
        \Omega^\big_R &:= \bigoplus_{q \in Q} \Omega^q_R \oplus \bigoplus_i \Omega^i_R,\\
        D^\big_R(\alpha) &:= \left(D^q_R(\alpha)\right)_{q \in Q} \oplus \left(D^i_R(\alpha)\right)_i.
    \end{align*}
\end{eg}

\begin{eg}
    Define $\Omega^\big_\Lambda := \Lambda^\big \otimes \left(\Z \oplus H^{2-*}(f,g) \right)$. 
    We denote the element $1 \otimes (1 \oplus 0)$ by $d\log T$, and an element $1 \otimes (0 \oplus \beta)$ by $d\nov_i$.
\end{eg}

\begin{defn}\label{def:conn}
    Given a derivation $D:R \to \Omega$, where $\Omega$ is a free $R$-module of finite rank, we define a \emph{connection} on an $R$-module $M$ to be a map
    $$\nabla:M \to \Omega \otimes M $$
    which commutes with the differential: $\nabla \circ d = d \circ \nabla$, and satisfies the Leibniz rule:
    $$\nabla(f \cdot m) = Df \otimes m + f \otimes \nabla(m).$$
    When we want to emphasize its dependence on the derivation $D$, we may refer to such $\nabla$ as a `$D$-connection', but we will not do this routinely.
\end{defn}

Connections on $M$ and $N$ induce connections on $M \otimes_\BbK N$ and $Hom_\BbK(M,N)$ (the latter uses the fact that $\Omega$ is free of finite rank to identify $Hom_R(M,\Omega \otimes N)$ with $\Omega \otimes Hom_R(M,N)$).
Connections on $M_i$ induce a connection on $\oplus_i M_i$ and $\prod_i M_i$ (again using the fact that $\Omega$ is free of finite rank to identify $\Omega \otimes \prod_i M_i$ with $\prod_i \Omega \otimes M_i$). 
A connection on $M$ induces one on the completion $\overline{M}$ (again using the fact that $\Omega$ is free of finite rank to identify $\Omega \otimes \overline{M}$ with $\overline{\Omega \otimes M}$). 

We define the Dubrovin--Givental $u$-connection 
\begin{align*}
    u\nabla^{DG}_R: QC^*(X;R^\big)[[u]] & \to \Omega^\big \otimes QC^*(X;R^\big)[[u]]\\
    u\nabla^{DG}_R(\alpha) & := uD^\big_R(\alpha) + \sum_q d\log \nov_q \otimes (PD(V_q) \star \alpha)  + \sum_i d\nov_i \otimes (\beta^i \star \alpha),
\end{align*}
where $PD(V_q) \star -$ (respectively $\beta^i \star -$) are certain chain maps which will be defined in Section \ref{sec:pdvq}, inducing the operation of quantum cup product with $PD(V_q)$ (respectively $\beta^i$) on the level of cohomology. 

Now let $f^*:R_1 \to R_2$ a morphism of coefficient rings, and $M$ an $R_1$-module. 
Then we may define the $R_2$-module $f^*M := R_2 \otimes_{R_1} M$. 
However, we need some extra structure in order to transfer a connection from $M$ to $f^*M$. 

\begin{defn}
Let $R_i$ be coefficient rings, and $D_i:R_i \to \Omega_i$ be derivations, for $i=1,2$. 
We define a morphism of coefficient rings with derivations, $f = (f^*,Df): (R_1,D_1) \to (R_2,D_2)$, to consist of:
\begin{itemize}
    \item A map of coefficient rings, $f^*:R_1 \to R_2$;
    \item A map of $R_2$-modules, $Df:R_2 \otimes_{R_1} \Omega_1  \to \Omega_2$, satisfying
    $$Df (r_2 \otimes D_1(r_1)) = r_2 D_2(f^*(r_1))$$ 
    for $r_i \in R_i$.
\end{itemize} 
\end{defn}

\begin{lem}\label{lem:pullback conn}
    Let $(f^*,Df):(R_1,D_1) \to (R_2,D_2)$ be a morphism of coefficient rings with derivations, and let $\nabla:M \to \Omega_1 \otimes M$ be a connection on the $R_1$-module $M$. 
    We may define a connection $f^*\nabla$ on the $R_2$-module $f^*M$, by
    \begin{align*}
    f^* \nabla: f^*M  & \to \Omega_2 \otimes f^*M \\
    f^*\nabla & := D_2 \otimes \id + Df \circ (\id \otimes \nabla). 
\end{align*}
\end{lem}
\begin{proof}
    This is routine from the definitions.
\end{proof}

\begin{lem}\label{lem:RLamder}
Consider the obvious derivation
$$D^\big_{\Lambda,\ge 0}: \Lambda^\big_{\ge 0} \to \Omega^\big_{\Lambda,\ge 0}:= \Lambda^\big_{\ge 0} \langle d\log T, d\novblam_i \rangle.$$
Suppose that we have
$$\kappa = \sum_{q \in Q} \kappa_q PD(V_q)$$
in $H^2(X,W)$. 
Then there is a morphism of coefficient rings with derivations, $(f^*,Df):(R^\big,D^\big_R) \to (\Lambda^\big_{\ge 0},D^\big_{\Lambda,\ge 0})$, where $f^*$ is the morphism \eqref{eq:RbigLambig}, and we have
\begin{align*}
    Df(d\log \nov_q) &:= \kappa_q d\log T \\
    Df(d\novb_j) &= \sum_i C^i_j d\novblam_i. 
\end{align*}
\end{lem}
\begin{proof}
    It suffices to check that
    $$Df(D^\big_R(\nov^u)) = D^\big_{\Lambda,\ge 0}(T^{\kappa(u)})$$
    for $u \in \NE$, and
    $$Df(D^\big_R(\novb_j)) = D^\big_{\Lambda,\ge 0}(\sum_i C^i_j \novblam_i).$$
    These are both routine from the definitions.
\end{proof}

\begin{lem}\label{lem:Rlamderconn}
    Let 
    $$(f^*,Df): (R^\big,\Lambda^\big) \to (\Lambda^\big_{\ge 0},\Omega^\big_{\Lambda,\ge 0})$$ 
    be the morphism of coefficient rings with derivations from Lemma \ref{lem:RLamder}. 
    Then 
    $$f^* \nabla^{DG}_R = \nabla^{DG}_\Lambda.$$
\end{lem}
\begin{proof}
    Again, this is routine from the definitions.
\end{proof}

\subsection{Getzler--Gauss--Manin connection}\label{sec:GGM}

Let $R$ be a coefficient ring with a derivation $D:R \to \Omega$, and $\cC$ an $R$-linear $A_\infty$ category. 
We define the operation
\[ B^{1|1}:CC^*(\cC) \otimes CC_*^{nu}(\cC) \to CC_*^{nu}(\cC)\]
\begin{align*}
 B^{1|1}(\varphi|c_0[c_1|\ldots|c_s]) &:= \sum_{j,k,\ell} e^+[c_{j+1}|\ldots|c_k|\varphi^*(\ldots,c_\ell)|\ldots,c_s,c_0,\ldots,c_j ],
\end{align*}
where $\sigma(B^{1|1}) \sigma(CC_*) = \sigma(CC_*) \sigma(e^+)\sigma(\varphi)$.

Suppose that all morphism spaces of $\cC$ admit $D$-connections $\tilde \nabla$ (this is automatically the case, for example, when they are free $R$-modules, as is the case for the big relative Fukaya category). 
We denote by $\tilde \nabla$ the connections induced on $fCC^*(\cC)$ and $fCC_*^-(\cC)$.  
We define the Getzler--Gauss--Manin connection on the chain level by 
\begin{align*}
u\nabla^{GGM}: fCC_*^{-}(\cC) & \to  \Omega \otimes fCC_*^{-}(\cC),\\
u\nabla^{GGM}(\alpha) &:= u\tilde \nabla(\alpha) - b^{1|1}(\tilde \nabla(\mu^*)|\alpha) - uB^{1|1}(\tilde\nabla(\mu^*)|\alpha). 
\end{align*} 
This descends to a connection
$u\nabla^{GGM}: fHC_*^-(\cC) \to \Omega \otimes fHC_*^-(\cC)$. 
It is shown in \cite{Sheridan_formulae} (in the unfiltered case, and over a field, but the argument applies verbatim in the present setting) that the connection is independent of the choice of connections $\tilde \nabla$, on the level of cohomology.

We expand out the terms of the connection on the summands $fCC_*^\vee$ and $fCC_*^\wedge$ of $fCC_*^{nu}$:
\begin{align*}
    b^{1|1}_{\vee\vee}(\tilde \nabla(\mu^*)|\alpha) &= b^{1|1}(\tilde\nabla(\mu^*)|\alpha)\\
    b^{1|1}_{\diamond\wedge} &= 0\\
    b^{1|1}_{\wedge\vee}(\tilde \nabla(\mu^*)|c_0[\ldots|c_s]) &= -\sum_{k} \tilde\nabla(\mu^*)(c_{k+1},\ldots,c_s)[c_0|\ldots|c_k]\\
    B^{1|1}_{\vee\wedge}(\tilde{\nabla}(\mu^*)|c_0[\ldots|c_s]) &= \sum_{j,k,\ell} c_{j+1}[\ldots|c_k|\tilde\nabla(\mu^*)(\ldots,c_\ell)|\ldots,c_s,c_0,\ldots,c_j ],\\
    B^{1|1}_{\diamond \diamond} &= 0 \qquad\text{for $\diamond\diamond \neq \vee\wedge$.}
\end{align*}

\begin{lem}\label{lem:GGM_C_Cbc}
The natural map $fCC_*^-(\cC^{\prebc}) \to fCC_*^-(\cC)$ from Lemma \ref{lem:HC_C_Cbc} respects connections, up to homotopy.
\end{lem}
\begin{proof}
    This follows by applying \cite[Theorem 3.32]{Sheridan_formulae} (adapted to the filtered setting) to the functor $F$ from \cite[Lemma 2.14]{relfukii}.
\end{proof}

\section{Recollection of the formalism for defining chain-level Floer theoretic operations}

For the reader's convenience, we give an informal summary of the formalism developed for defining chain-level Floer-theoretic operations in \cite{relfukii}, to which we refer for all details.

\subsection{Systems of families of domains}

A \emph{mixed curve} consists of a Riemann surface with boundary, and interior and boundary marked points, together with a union of intervals and rays attached at their boundary points to interior marked points of the curve. 
The interior marked points are designated as either `stabilizing' or `bulk'; the bulk marked points get an endpoint of a ray or interval attached to them. 

Given a stable topological type $\top$ of mixed curves, there is a moduli spaces $\cC(\top)$ of stable mixed curves of that topological type, parametrizing the complex structure on the curve as well as the lengths of the intervals (which may lie in $[0,\infty]$); and it has a Deligne--Mumford compactification $\Cbar(\top)$.

A \emph{choice of directions} for a mixed curve consists of a designation of the boundary marked points and nodes, as well as bulk marked points, as `incoming' or `outgoing'. 

A \emph{family of domains} $\Rbar(F)$ can be thought of as a `stratified chain' in one of the Deligne--Mumford moduli spaces of stable mixed curves, with extra decorations. 
In slightly more detail, it consists of a certain kind of stratified space called a `semianalytic pseudomanifold with boundary', $\Rbar(F)$, together with a stratified map $\Rbar(F) \to \Cbar(\top_F)$ for some topological type of mixed curve $\top_F$, a continuous family of choices of directions for the curves parametrized by $\Rbar(F)$, and a subset $P^{sym}(s)$ of the set of stabilizing marked points, for each stratum $s$ of $\Rbar(F)$, with the property that $s \subseteq t \Rightarrow P^{sym}(s) \supseteq P^{sym}(t)$. 
The stabilizing points contained in $P^{sym}(s)$ are called `symmetric', while the other stabilizing points are called `non-symmetric'. 
We require that the subgroup $Sym(s) \subseteq Sym(P^{sym}(s))$, whose action on the topological type $\top_s$ by relabelling marked points preserves the topological type, acts on $s$ by an orientation-preserving diffeomorphism.

There are various natural operations on families of domains: we may take the disjoint union of several families; their boundary attachment by joining an incoming boundary marked point to an outgoing boundary marked point; their interior attachment with length parameter $0$, $\infty$, or $[0,\infty]$, which joins one bulk marked point to another by an interval of length $0$, $\infty$, or by the family of all intervals of lengths in $[0,\infty]$; we may pass to a boundary stratum of $\Rbar(F)$; and we may stabilize the family by adding incoming boundary marked points, stabilizing marked points, or incoming bulk marked points (where the positions of these additional marked points are unconstrained). 

We define a \emph{system of families of domains} to be a set of families of domains which is closed under all of the above operations. We examine what this means in the example of the $A_\infty$ operations which define the big relative Fukaya category, following \cite[Section 5.2]{relfukii}. 
For any $s \in \Z_{\ge 0}$ and $\ell = (\ell_\bulk,\ell_\stab) \in (\Z_{\ge 0})^2$, such that $s+2\ell \ge 3$, we define $\top_{\mu,s,\ell}$ to be the topological type consisting of a disc with $s+1$ cyclically ordered boundary marked points, of which the first is outgoing and the rest are incoming; $\ell_\bulk$ bulk marked points, all of which are incoming; and $\ell_\stab$ stabilizing marked points, all of which are symmetric on all boundary strata. 
We define $\Rbar(\mu,s,\ell) \to \Cbar(\top_{\mu,s,\ell})$ to be the fundamental chain, i.e., the identity map. 

The family $\Rbar(\mu,s,\ell)$ has codimension-$1$ boundary strata consisting of nodal discs; each of these is obtained by boundary attachment from a disjoint union of two families $\Rbar(\mu,s_1,\ell_1)$ and $\Rbar(\mu,s_2,\ell_2)$. 
In fact, the families $\Rbar(\mu,s,\ell)$ may also be considered as stabilizations of the empty family; we refer to \cite[Section 3.1]{relfukii} for explanation on this point.

\subsection{Floer-theoretic operations}

Given a family of domains $\Rbar(F) \to \Cbar(\top_F)$, we define a `Lagrangian labelling' $\bL$ to associate a Lagrangian brane $L_C$ to each boundary component of the curve after all boundary marked points are punctured, which `match' at the boundary nodes. 
For any incoming (respectively outgoing) boundary marked point $p$, we define $L_p^-/L_p^+$ to be the Lagrangian labels in positive/negative (respectively negative/positive) direction from $p$. 
We define
\begin{multline}
B^\ex(F,\bL) := Hom^*\left(\S(F) \otimes \bigotimes_{p \in P^{bulk,in}} \sigma(-2)QC^*(X;\Z) \otimes \bigotimes_{p \in P^{\partial,in}} \fuk^{\ex}(L_p^-,L_p^+) ,\right.\\
\left.\bigotimes_{p \in P^{bulk,out}} QC^*(X;\Z) \otimes  \bigotimes_{p \in P^{\partial,out}} \fuk^{\ex}(L_p^-,L_p^+)\right),
 \end{multline}
 where $\S(F)$ is a certain graded $\Z_2$-torsor associated to the family $\Rbar(F)$, $P^{bulk,in/out}$ is the set of incoming/outgoing bulk marked points, $P^{\partial,in/out}$ the set of incoming/outgoing boundary marked points. 

Associated to a family of domains $\Rbar(F)$ with at most one non-symmetric stabilizing marked point, and satisfying a condition called `$f^{sym}$-stability', together with a choice of Lagrangian labelling $\bL$, a homotopy class $A$ of maps from the domain into $X$, and an element $q \in Q$ in the case that there is a non-symmetric stabilizing marked point in the family, we define an element
$$F^\ex_{\bL,A,(q)} \in B^\ex(F) \otimes \sigma(\mu(A))$$
by counting the $0$-dimensional components of the moduli space of pseudoholomorphic maps from a mixed curves with domain parametrized by an element of $\Rbar(F)$, boundary constrained to lie on the Lagrangians $\bL$, homotopy class $A$, and the non-symmetric stabilizing marked point (if it exists) constrained to lie on $V_q$. 

We then define
$$F^\sm_{\bL,(q)} := \sum_A \nov^A \cdot F^\ex_{\bL,A,(q)} \in R^\sm \otimes B^\ex(F).$$
It is straightforward to verify that these operations behave as expected under disjoint union, boundary attachments, and interior attachments with length parameter $\infty$ (see \cite[Lemmas 4.16, 4.17]{relfukii}).  
We also show that if $\Rbar(F)$ is a $f^{sym}$-stable family of domains, all of whose codimension-$1$ boundary strata are $f^{sym}$-stable, then 
\begin{equation}\label{eq:codim 1}\partial(F^\sm_{\bL,(q)}) = \sum_{F'} (F')^\sm_{\bL,(q)}
\end{equation}
where the sum is over all codimension-$1$ boundary components $\Rbar(F')$ of $\Rbar(F)$ (see \cite[Lemma 4.18]{relfukii}). 

Finally, given a family of domains $\Rbar(F)$, and a subset $P \subset P^{\bulk,in}$ of the incoming bulk marked points, we define 
\begin{multline}
B^\big(F,P) := Hom\left(\S(F,P) \otimes \bigotimes_{p \in P} QC^*(X;R^{big}) \otimes \bigotimes_{p \in P^{\partial,in}} \cF^{big}(L_p^-,L_p^+) ,\right.\\
\left.\bigotimes_{p \in P^{bulk,out}} QC^*(X;R^{big}) \otimes  \bigotimes_{p \in P^{\partial,out}} \cF^{big}(L_p^-,L_p^+)\right)
 \end{multline}
 where $\S(F,P) = \S(F)\sigma(-2|P|)$. 
 
There is a natural map
$$B^\sm(F) \otimes_{R^{small}} R^{big} \to B^\big(F,P),$$
induced by the map
\begin{align*}
    \bigotimes_{p \in P^{bulk,in} \setminus P} \sigma(2)QC^*(X;R^{big})^\vee & \to R^{big} \\
    p_1 \otimes \ldots \otimes p_\ell &\mapsto \frac{1}{\ell!} \novb_1 \ldots \novb_\ell.
\end{align*}
We denote the image of $F^\sm_{\bL,(q)} \otimes 1$ under this map by $F^\big_{P,\bL,(q)}$. 

The operations we define in this paper are all constructed to be equal to $F^\big_{P,\bL,(q)}$ for an appropriate family $\Rbar(F)$; and we prove relations between them using \eqref{eq:codim 1}.

\section{Weak proper Calabi--Yau structure}\label{sec:wpcy}

In this section we construct the weak proper Calabi--Yau structure on $\fuk^\sm(X,D;\Lambda^\sm)^\bc$, proving Theorem \ref{thm:wpcy}. 

We define $\Rbar(wpCY,0,0)$ to be the subspace of the moduli space of discs with two incoming boundary marked points, $p_0^\partial$ and $p_1^\partial$, and one interior marked point $p_1^{int}$, where the disc can be parametrized as the unit disc with $p_0^\partial$ lying at $-1$, $p_1^\partial$ lying at $+1$, and $p_1^{int}$ lying at $0$. Let $\Rbar(wpCY,s,\ell)$ be its stabilization, where $s=(s_1,s_2)$ records the number of boundary marked points with negative imaginary part, and with positive imaginary part respectively, and $\ell=(\ell_\bulk,\ell_\stab)$ records the number of bulk and symmetric stabilizing marked points respectively. 

We define $\Rbar(wpCY^{-1},s,\ell)$ to be the same as $\Rbar(wpCY,s,\ell)$ except that the boundary marked points $p_0^\partial$ and $p_1^\partial$ are outgoing rather than incoming. 

We define $\Rbar(H^1_{wpCY},0,0)$ to be the subspace of the moduli space of discs with one incoming boundary marked point, $p_0^\partial$, one outgoing boundary marked, $p_1^\partial$, and two interior marked points $p_1^{int}$ and $p_2^{int}$, where the disc can be parametrized as the unit disc with $p_0^\partial$ lying at $-1$, $p_1^\partial$ lying at $+1$, $p_1^{int}$ at $-t$, and $p_2^{int}$ at $+t$, for $t \in [0,1]$. 
It has two boundary components: one, at $t=1$, is identified with the boundary gluing of $\Rbar(wpCY,0,0)$ with $\Rbar(wpCY^{-1},0,0)$, and the other, at $t=0$, is the moduli space $\Rbar(wpCY_{\id},0,0)$ of stable discs with one incoming boundary marked point $p_0^\partial$, one outgoing boundary marked point $p_1^\partial$, and a sphere with marked points $p_1^{int}$ and $p_2^{int}$, attached at an interior node. 
We define $\Rbar(H^1_{wpCY},s,\ell)$ to be its stabilization.

We define $\Rbar(H^2_{wpCY},s,\ell)$ to be the same as $\Rbar(H^1_{wpCY},s,\ell)$, except that the two boundary marked points are outgoing rather than incoming.

\begin{thm}\label{thm:wpcy_big}
    There exist:
    \begin{itemize}
        \item a morphism 
                $$wpCY \in hom_{\bimod{\fuk^\big}{\fuk^\big}}(\sigma(n) \fuk^\big_\Delta,(\fuk^\big_\Delta)^\vee)$$
                which is closed, i.e. $\partial(wpCY) = 0$;
        \item a morphism 
                $$wpCY^{-1} \in hom_{\bimod{\fuk^\big}{\fuk^\big}}((\fuk^\big_\Delta)^\vee,\sigma(n) \fuk^\big_\Delta)$$
                which is closed, i.e. $\partial(wpCY^{-1}) = 0$;
        \item an element 
        $$H^1_{wpCY} \in hom^{-1}_{\bimod{\fuk^\big}{\fuk^\big}}(\fuk^\big_\Delta,\fuk^\big_\Delta)$$
        satisfying 
        \begin{equation}\label{eq:H1wpCY}
        wpCY^{-1} \circ wpCY = \id + \partial(H^1_{wpCY});
        \end{equation}
        \item an element 
        $$H^2_{wpCY} \in hom^{-1}_{\bimod{\fuk^\big}{\fuk^\big}}((\fuk^\big_\Delta)^\vee,(\fuk^\big_\Delta)^\vee)$$
        satisfying
        \begin{equation}\label{eq:H2wpCY}
        wpCY \circ wpCY^{-1} = \id + \partial(H^2_{wpCY}).
        \end{equation}
        \end{itemize}
    In particular, the morphism
    $$[wpCY] \in hom_{H(\bimod{\fuk^\big}{\fuk^\big})}(\sigma(n) \fuk^\big_\Delta,(\fuk^\big_\Delta)^\vee)$$ 
    is an isomorphism, and hence defines an $n$-wpCY structure on $\fuk^\big(X,D)$.
\end{thm}
\begin{proof}
    We define a bimodule morphism
    $$wpCY: \sigma(n) \fuk^\big(X,D)_\Delta \to \fuk^\big(X,D)^\vee_\Delta$$
    by setting
    $$wpCY^{s_1|1|s_2} = \sum_{\ell,\bL} (wpCY,(s_1,s_2),\ell)_{\bL}(e). $$
    It is closed, by \cite[Lemma 4.18]{relfukii} applied to the family $(wpCY,s,\ell)$, and hence defines a morphism $[wpCY]$ in $H(\bimod{\fuk^\big(X,D)}{\fuk^\big(X,D)})$.
    
    We similarly define a morphism $[wpCY^{-1}]$ in the opposite direction, and show it is closed, using the family $(wpCY^{-1},s,\ell)$.

    Applying \cite[Lemma 4.18]{relfukii} to the family $(H^1_{wpCY},s,\ell)$ yields the equation
    $$wpCY^{-1} \circ wpCY = wpCY_{\id} + \partial(H^1_{wpCY})$$
    in $hom_{\bimod{\fuk^\big}{\fuk^\big}}(\fuk^\big_\Delta,\fuk^\big_\Delta)$. 
    For the family $(wpCY_{\id},s,\ell)$, we may choose perturbation data to be independent of the position at which the nodal sphere is attached; this ensures that the only contribution comes from strips which are constant along their length, so that $wpCY_{\id}$ is equal to the identity endomorphism (compare \cite[Theorem 5.8]{relfukii}).
    This completes the proof of \eqref{eq:H1wpCY}.
    The proof of \eqref{eq:H2wpCY} is analogous.
\end{proof}

We now define $\Rbar(H^3_{wpCY},0,0)$ to be the subspace of the moduli space of discs with two incoming boundary marked points $p_0^\partial$ and $p_1^\partial$, and one interior marked point $p_1^{int}$, where the disc can be parametrized as the unit disc with $p_0^\partial$ lying at $-1$, $p_1^\partial$ lying at $+1$, and $p_1^{int}$ lying at $-ti$ for $t \in [0,1]$. 
It has two boundary components: one, at $t=0$, is identified with $\Rbar(wpCY,0,0)$; the other, at $t=1$, is identified with the attachment of $\Rbar(\cO\cC,0,0)$ with $\Rbar(\mu,2,0)$. 
We define $\Rbar(H^3_{wpCY},s,\ell)$ to be its stabilization. 

\begin{lem}\label{lem:wpcy_oc}
    Let $\varphi \in (\sigma(n)fCC_*(\fuk^\big))^\vee$ be the element
    $\varphi(\alpha) := \langle \cO\cC^\big(\alpha),e\rangle$. 
    Then there exists $H^3_{wpCY} \in hom_{\bimod{\fuk^\big}{\fuk^\big}}(\sigma(n)\fuk^\big_\Delta, (\fuk^\big)^\vee_\Delta)$ such that
    $$G_{\fuk^\big}(\varphi) = wpCY + \partial(H^3_{wpCY}).$$
\end{lem}
\begin{proof}
    Follows by applying \cite[Theorem 4.18]{relfukii} to the family $\Rbar(H^3_{wpCY},s,\ell)$.
\end{proof}

\begin{proof}[Proof of Theorem \ref{thm:wpcy}]
    Consider the morphisms $wpCY$, $wpCY^{-1}$, $H^1_{wpCY}$ and $H^2_{wpCY}$ from Theorem \ref{thm:wpcy_big}. 
    Base-changing them to $\Lambda^\sm_{\ge 0}$ gives an $n$-wpCY structure on $\fuk^\sm(X,D;\Lambda^\sm_{\ge 0})$. 
    Passing to bounding cochains gives an $n$-wpCY structure on $\fuk^\sm(X,D;\Lambda^\sm)^\bc$, by Lemma \ref{lem:wpcy_bc}. 

    By Lemma \ref{lem:wpcy_oc}, base-changed to $\Lambda^\sm_{\ge 0}$, the $n$-wpCY structure on $\fuk^\sm(X,D;\Lambda^\sm)$ is given by $G_{\fuk^\sm}(\varphi)$; and hence, by Lemma \ref{lem:wpcy_gc}, the $n$-wpCY structure on $\fuk^\sm(X,D;\Lambda^\sm)^\bc$ is given by $G_{\fuk^\sm(X,D;\Lambda^\sm)^\bc}(\phi)$. 
    This completes the proof.
\end{proof}

\section{Cyclic open--closed map}\label{sec:cyc oc}

In this section we construct the cyclic open--closed map, proving Theorem \ref{thm:OC_minus_bc}.

Let $D^2$ be the unit disc in $\C$, and $S^1$ its boundary. 
Define $\Rbar(\cO\cC,s,\ell)$ be the family introduced in \cite[Section 5.5]{relfukii}. 
There are analytic maps 
$$\theta_j: \Rbar(\cO\cC,s,\ell) \to S^1,$$
for $j=0,\ldots,s$, which sends a disc to the position of the $j$th marked boundary point on $S^1$, if the disc is parametrized so that the outgoing interior marked point lies at the origin and $\theta_0$ is constant equal to $1 \in S^1$. Let $r:\Rbar(\cO\cC,s,\ell) \to \Rbar(\cO\cC,s,\ell)$ be the automorphism which relabels the boundary marked points, so $r^*\theta_j = \theta_{j-1}$.

We now introduce a set $K = \mathbb{N}_0 \times \{\vee,\wedge\}$, and we will denote an element $(k,\vee)$ (respectively, $(k,\wedge)$) of this set by $k^\vee$ (respectively, $k^\wedge$).

We define 
$$D_k = \{(z_1,\ldots,z_k) \in \C^k: \sum_{i=1}^k |z_i|^2 \le 1\}.$$
We inductively define a stratification on 
$$\Rbar({\cO\cC},k^\vee,s,\ell) := \Rbar(\cO\cC,s,\ell) \times D_k$$
by identifying 
$$r^j \cdot \Rbar({\cO\cC},(k-1)^\vee,s,\ell) \subset  \Rbar({\cO\cC}^{S^1},(k-1)^\wedge,s,\ell) := \Rbar(\cO\cC,s,\ell) \times \partial D_k$$ 
as the subset where $\arg(z_k) = \theta_j$,\footnote{We define ``$\arg(z) = \theta$'' to mean ``$z=0$, or $z \neq 0$ and $\arg(z) = \theta$''.} and defining $r^j \cdot \Rbar({\cO\cC},(k-1)^\wedge,s,\ell) \subset \Rbar(\cO\cC,s,\ell) \times \partial D_k$ as the subset where $\theta_j \le \arg(z_k) \le \theta_{j+1}$. 
Note that we can identify $r^j \cdot \Rbar({\cO\cC},(k-1)^\vee,s,\ell)$ with $\Rbar({\cO\cC},(k-1)^\vee,s,\ell) = \Rbar(\cO\cC,s,\ell) \times D_{k-1}$ via projection to the first $k-1$ factors of $D_k$.

We need to specify the identifications of boundary strata for these families:
\begin{itemize}
\item We specify that $\Rbar({\cO\cC},0^\vee,s,\ell)$ is identified with $ \Rbar(\cO\cC,s,\ell)$. 
\item We specify that $\Rbar({\cO\cC},k^\vee,s,\ell)$ is a stabilization of $\Rbar({\cO\cC},k^\vee,0,0)$, with all stabilizing marked points symmetric. 
\item We specify that $\Rbar({\cO\cC}^{S^1},k^\wedge,s,\ell)$ is a stabilization of $\Rbar({\cO\cC},k^\wedge,1,0)$. (Recall that this imposes conditions on the boundary identifications of disc and sphere bubbles; this implies corresponding boundary identifications for the sub-family $\Rbar({\cO\cC},k^\wedge,s,\ell)$.) 
\item We identify $\Rbar({\cO\cC},k^\vee,s,\ell)$ and $r \cdot \Rbar({\cO\cC},k^\vee,s,\ell)$ as codimension-$1$ boundary components of $\Rbar({\cO\cC},k^\wedge,s,\ell)$.
\item We identify $r^i \cdot \Rbar({\cO\cC},(k-1)^\wedge,s,\ell)$ as a codimension-$1$ boundary component of $\Rbar({\cO\cC},k^\vee,s,\ell)$, for each $0 \le i \le s$.  
\end{itemize}

\begin{rem}
    To relate these families with the corresponding ones in \cite[Section 5.5]{Ganatra:cyclic}, observe that there are maps
    \begin{align*}
        _k \overline{\check{\cR}}^1_s & \to \Rbar({\cO\cC},k^\vee,s,0,0)\\
        _k \overline{\hat{\cR}}^1_s & \to \Rbar({\cO\cC},k^\wedge,s,0,0)
    \end{align*}
    defined by sending
\begin{align*}
    (p_1,\ldots,p_k) &\mapsto (z_1,\ldots,z_k) \qquad \text{where}\\
    z_i & = \begin{cases} p_1& \text{if $i=1$,} \\
    \left( 1-\frac{|p_{i-1}|^2}{|p_i|^2} \right) p_i& \text{if $i>1$.}
    \end{cases}
\end{align*}
Note that the `degenerate' strata where $|p_{i-1}| = |p_i|$ in \cite[Section 5.5]{Ganatra:cyclic} get sent to the locus where $z_i = 0$ in the new version, which is not a stratum; so these strata have been collapsed and the stratification coarsened to remove them. 
\end{rem}

\begin{lem}\label{lem:oc cyc chain}
    There are maps
    \begin{align*}
        \cO\cC^\vee_k: \sigma(\cO\cC)\sigma(-2k)fCC^{\vee}_*(\fuk^\big(X,D)) & \to QH^*(X;R^\big)\\
        \cO\cC^\wedge_k: \sigma(\cO\cC)\sigma(-2k)fCC^{\wedge}_*(\fuk^\big(X,D)) & \to QH^*(X;R^\big)
    \end{align*}
    satisfying:
    \begin{align}
    \label{eq:OC check rel}    \cO\cC^\vee_k \circ b + \cO\cC^\wedge_{k-1} \circ B & = \partial \circ \cO\cC^\vee_k \\
    \label{eq:OC hat rel}    \cO\cC^\wedge_k \circ b_{\wedge\wedge} + \cO\cC^\vee_k \circ b_{\wedge\vee} & = \partial \circ \cO\cC^\wedge_k,
    \end{align}
    where $\sigma(b) = \sigma(2)\sigma(B) = \sigma(\partial)$, and furthermore,
    \begin{align*}
        \cO\cC^\vee_0|_{fCC_*(\fuk^\big(X,D))}&= \cO\cC^\big.
    \end{align*}
\end{lem}
\begin{proof}
    Follows from \cite[Lemma 4.18]{relfukii}, applied to the families $\cO\cC^\vee$ and $\cO\cC^\wedge$; the correspondence between codimension-one boundary components and terms in the equations is analogous to \cite[Proposition 5.17]{Ganatra:cyclic}.
\end{proof}

\begin{cor}
    The map 
    \begin{align*}
        \cO\cC^{-,\big}: \sigma(\cO\cC)fCC_*^-(\fuk^\big(X,D)) & \to QH^*(X;R^\big)[[u]] \\
        \cO\cC^{-,\big}(\alpha) & := \sum_{k=0}^\infty u^k\cdot \left(\cO\cC^\vee_k(\alpha^\vee) + \cO\cC^\wedge_k(\alpha^\wedge)\right),
    \end{align*}
    extended $u$-linearly, where $\alpha = (\alpha^\vee,\alpha^\wedge) \in fCC^{\vee}_* \oplus fCC^{\wedge}_*$, is a chain map.
\end{cor}

\begin{proof}[Proof of Theorem \ref{thm:OC_minus_bc}]
    By tensoring the map $\cO\cC^{-,\big}$ with $\Lambda^\big_{\ge 0}$, along $R^\big \to \Lambda^\big_{\ge 0}$, we obtain a chain map
    $$\sigma(\cO\cC)fCC_*^-(\fuk^\big(X,D;\Lambda^\big_{\ge 0})) \to QH^*(X;\Lambda^\big_{\ge 0})[[u]].$$
    Next, pre-composing with the chain map $F_*$ from Lemma \ref{lem:HC_C_Cbc}, we obtain a chain map
    $$\sigma(\cO\cC)fCC_*^-(\fuk^\big(X,D;\Lambda^\big_{\ge 0})^{\prebc}) \to QH^*(X;\Lambda^\big_{\ge 0})[[u]].$$ 
    Finally, the desired chain map is obtained by restricting to the subcategory of small bounding cochains, and tensoring with $\Lambda^\big$ along $\Lambda^\big_{\ge 0} \to \Lambda^\big$. 
    The fact that it extends $\cO\cC^{-,\big}$ follows immediately from the final point in Lemma \ref{lem:oc cyc chain}.
\end{proof}

\section{Cyclic open--closed map respects pairings}\label{sec:oc pair}

In this section we prove that the cyclic open--closed map respects pairings, proving Theorems \ref{thm:mukai_int} and \ref{thm:oc_hres_bc}.

We recall that the quantum pairing
$$\langle-,-\rangle: QC^*(X;\Z) \otimes QC^*(X;\Z) \to \Z$$
is defined by counting an appropriate moduli space of flowlines, as in \cite[Section 4.1]{relfukii}; we extend it $R^\big$-linearly and $u$-sesquilinearly to the $R^\big[[u]]$-sesquilinear pairing
$$\langle -, - \rangle : QC^*(X;R^\big)[[u]] \otimes_{R^\big[[u]]} QC^*(X;R^\big)[[u]] \to R^\big[[u]]. $$

We define the family $\Rbar(H^{12}_{\langle k \rangle},s,\ell)$, for $k=(k^{\diamond_{in}}_{in},k^{\diamond_{out}}_{out}) \in K \times K$, $s=(s_{in},s_{out})$, and $\ell = (\ell_{\bulk,in},\ell_{\bulk,out},\ell_{\stab,in},\ell_{\stab,out})$ to be obtained by attaching $\Rbar(\cO\cC,\bar{k}_{in},s_{in},(\ell_{\bulk,in},\ell_{\stab,in}))$ to $\Rbar(\cO\cC,\bar{k}_{out},s_{out},(\ell_{\bulk,out},\ell_{\bulk,out}))$ at their outgoing bulk marked point, with length parameter $0$. 

We similarly define the versions with $12$ replaced by $1$, which is the same except the length parameter is $[0,\infty]$.

\begin{lem}\label{lem:cardy triv}
    There are maps
    \begin{align*}
        H^{12}_{\langle k_{in}^{\diamond_{in}},k_{out}^{\diamond_{out}}\rangle}: fCC^{\diamond_{in}}_*(\fuk^\big(X,D))   \otimes fCC^{\diamond_{out}}_*(\fuk^\big(X,D)) &  \to R^{\big}
        \end{align*}
    with 
    $$\sigma(H^{12}_{\langle k_{in}^{\diamond_{in}},k_{out}^{\diamond_{out}}\rangle}) = \sigma_{in}(\cO\cC)\sigma_{out}(\cO\cC) \sigma(-2n-2k_{in}-2k_{out}),$$
    and
    \begin{align*}
        H^{1}_{\langle k_{in}^{\diamond_{in}},k_{out}^{\diamond_{out}}\rangle}: fCC^{\diamond_{in}}_*(\fuk^\big(X,D)) \otimes fCC^{\diamond_{out}}_*(\fuk^\big(X,D)) & \to R^\big
    \end{align*}
    with 
    $$\sigma(H^{1}_{\langle k_{in}^{\diamond_{in}},k_{out}^{\diamond_{out}}\rangle}) = \sigma(\partial)^\vee \sigma(H^{12}_{\langle k_{in}^{\diamond_{in}},k_{out}^{\diamond_{out}}\rangle}),$$
    for all $(k_{in}^{\diamond_{in}},k_{out}^{\diamond_{out}}) \in K \times K$, satisfying
    \begin{multline}
        \langle \cO\cC^\vee_{k_{in}}(\alpha),\cO\cC^\vee_{k_{out}}(\beta) \rangle - H^{12}_{\langle k^\vee_{in},k^\vee_{out}\rangle}(\alpha,\beta) = \partial(H^1_{\langle k^\vee_{in},k^\vee_{out}\rangle}(\alpha,\beta)) \\+ H^1_{\langle k^\vee_{in},k^\vee_{out}\rangle}(b\alpha,\beta) + H^1_{\langle k^\vee_{in},k^\vee_{out}\rangle}(\alpha,b\beta) + H^1_{\langle (k_{in}-1)^\wedge,k^\vee_{out}\rangle}(B\alpha,\beta) + H^1_{\langle k^\vee_{in},(k_{out}-1)^\wedge\rangle}(\alpha,B\beta),
    \end{multline}
    \begin{multline}\label{eq:hat check}
        \langle \cO\cC^\wedge_{k_{in}}(\alpha),\cO\cC^\vee_{k_{out}}(\beta) \rangle - H^{12}_{\langle k^\wedge_{in},k^\vee_{out}\rangle}(\alpha,\beta) = \partial(H^1_{\langle k^\wedge_{in},k^\vee_{out}\rangle}(\alpha,\beta)) \\+ H^1_{\langle k^\wedge_{in},k^\vee_{out}\rangle}(b_{\wedge\wedge}\alpha,\beta) + H^1_{\langle k^\vee_{in},k^\vee_{out}\rangle}(b_{\wedge\vee}\alpha,\beta) + H^1_{\langle k^\wedge_{in},k^\vee_{out}\rangle}(\alpha,b\beta) + H^1_{\langle (k_{in}-1)^\vee,k^\wedge_{out}\rangle}(\alpha,B\beta),
    \end{multline}
    the analogue of $\eqref{eq:hat check}$ with $\wedge$ and $\vee$ swapped, and
    \begin{multline}
        \langle \cO\cC^\wedge_{k_{in}}(\alpha),\cO\cC^\wedge_{k_{out}}(\beta) \rangle - H^{12}_{\langle  k^\wedge_{in},k^\wedge_{out}\rangle}(\alpha,\beta) = \partial(H^1_{\langle k^\wedge_{in},k^\wedge_{out}\rangle}(\alpha,\beta)) \\+ H^1_{\langle k^\wedge_{in},k^\wedge_{out}\rangle}(b_{\wedge\wedge}\alpha,\beta) + H^1_{\langle k^\vee_{in},k^\wedge_{out}\rangle}(b_{\wedge\vee}\alpha,\beta) + H^1_{\langle k^\wedge_{in},k^\wedge_{out}\rangle}(\alpha,b_{\wedge\wedge}\beta) + H^1_{\langle k^\wedge_{in},k^\vee_{out}\rangle}(\alpha,b_{\wedge\vee}\beta),
    \end{multline}
    where $\sigma(\partial) = \sigma(b) = \sigma(2)\sigma(B)$.
\end{lem}
\begin{proof}
    This follows by applying \cite[Lemma 4.18]{relfukii} to the families $H^1_{\langle\rangle}$. 
\end{proof}

\begin{cor}\label{cor:cardy 1}
    There are sesquilinear maps
    \begin{align*}
        H^{12}_{\langle\rangle}: \sigma_{in}(\cO\cC)\sigma_{out}(\cO\cC)fCC_*^-(\fuk^\big) \otimes fCC_*^-(\fuk^\big) & \to R^\big[[u]],\\
        H^1_{\langle\rangle}: \sigma(\partial)^\vee\sigma_{in}(n)\sigma_{out}(n)fCC_*^-(\fuk^\big) \otimes fCC_*^-(\fuk^\big) & \to R^\big[[u]],
    \end{align*}
    satisfying
    \begin{multline}
        \langle \cO\cC^{\big,-}(\alpha),\cO\cC^{\big,-}(\beta)\rangle - H^{12}_{\langle\rangle}(\alpha,\beta) \\= \partial(H^1_{\langle\rangle}(\alpha,\beta)) + H^1_{\langle\rangle}((b+uB)(\alpha),\beta) + H^1_{\langle\rangle}(\alpha,(b+uB)(\beta)).
    \end{multline}
\end{cor}
\begin{proof}
    We define
    \begin{multline}\label{eq:put the us}
        H^{1/12}_{\langle\rangle}(\alpha,\beta) = \sum_{k_{in},k_{out}} u^{k_{in}}\cdot(-u)^{k_{out}} \cdot (H^{1/12}_{\langle k^\vee_{in},k^\vee_{out}\rangle}(\alpha^\vee,\beta^\vee) \\
        + H^{1/12}_{\langle k^\wedge_{in},k^\vee_{out}\rangle}(\alpha^\wedge,\beta^\vee) + H^{1/12}_{\langle k^\vee_{in},k^\wedge_{out}\rangle}(\alpha^\vee,\beta^\wedge) + H^{1/12}_{\langle k^\wedge_{in},k^\wedge_{out}\rangle}(\alpha^\wedge,\beta^\wedge))
    \end{multline}
    for $\alpha$, $\beta \in fCC_*^{nu}$, and extend $u$-sesquilinearly. 
    The claim then follows from Lemma \ref{lem:cardy triv}. 
\end{proof}

We now consider the family $\Rbar(A,0,0)$ of annuli with no interior marked points, one incoming marked point $p_{0,in}^\partial$ on the inner boundary component, and one incoming marked point $p_{0,out}^\partial$ on the outer boundary component. 
Annuli in the open locus $r \in \cR(A,0,0)$ may be uniquely parametrized as $\{z \in \C: R \le |z| \le 1\}$, so that $p_{0,out}^\partial$ corresponds to the point $1$ and $p_{0,in}^\partial$ corresponds to the point $Re^{i\theta}$, with $R \in (0,1)$ and $e^{i\theta} \in S^1$. 
The Deligne--Mumford compactification adds a point over $R=0$ and a circle over $R=1$, so that $\Rbar(A,0,0)$ is a disc (but we observe that, while the parametrization $Re^{i\theta}$ extends over $R=0$, it does not extend over $R=1$; rather, we take the disc parametrized by $Re^{i\theta}$, perform a real blowup at $1$, then a real blowdown of the proper transform of the boundary of the disc).

We define $\Rbar(A,s,\ell)$ to be its stabilization, where $s=(s_{in},s_{out})$ records the number of incoming boundary marked points on the inner and outer boundary components respectively, in addition to $p_{0,in}^\partial$ and $p_{0,out}^\partial$, and $\ell = (\ell_{\bulk},\ell_{\stab})$ records the number of bulk and symmetric stabilizing marked points. 
There is an analytic map $R:\Rbar(A,s,\ell) \to [0,1]$, which is the pullback of the parameter $R$ defined on $\Rbar(A,0,0)$. 
The strata lying inside $R^{-1}(0)$ consist of `circles of discs' with disc and sphere bubbles attached; while those inside $R^{-1}(1)$ consist of discs connected by a chain of spheres, with disc and sphere bubbles attached. 

We have natural analytic functions $\theta_j^{in}:\Rbar(A,s,\ell) \to S^1$, for $0 \le j \le s_{in}$, which are defined to be $\arg(p_{0,in}^\partial) - \arg(p_{j,in}^\partial)$, the difference in arguments of the respective boundary marked points on the interior boundary. 
We similarly have analytic functions $\theta_j^{out}:\Rbar(A,s,\ell) \to S^1$, for $0 \le j \le s_{out}$, which are defined to be $\arg(p_{j,out}^\partial) - \arg(p_{0,out}^\partial)$. 
We have $\theta = \arg(p_{0,in}^\partial) - \arg(p_{0,out}^\partial)$, for $R \in (0,1)$.

Note that the codimension-$2$ boundary strata lying over $\{R=0\}$ are isomorphic to $\Rbar(H^{12}_{\langle k_{in}^\vee,k_{out}^\vee\rangle},s,\ell)$; and furthermore, the restriction of $\theta_{j,in/out}$ to this stratum is equal to the pullback of $\theta_j$ from $\Rbar(\cO\cC_{in/out},s_{in/out},\ell_{in/out})$.
In order for this property to hold, we are forced to use the opposite sign in the conventions in the definitions of $\theta_j^{in}$ and $\theta_j^{out}$ above. 

For $k_{in},k_{out} \ge 0$, we consider the space 
$$D_{k_{in},k_{out}} := \left\{ (z^{in},z^{out},R) \in \C^{k_{in}} \times \C^{k_{out}} \times [0,1]: \|z^{in}\|^2 \le R, \|z^{out}\|^2 \le R \right\}.$$
We define a family 
$$\Rbar(H^3_{\langle k^\vee_{in},k^\vee_{out}\rangle},s,\ell):= \Rbar(A,s,\ell) \times_{[0,1]_R} D_{k_{in},k_{out}}.$$
We define a sub-family 
$$\Rbar(H^2_{\langle k^\vee_{in},k^\vee_{out}\rangle},s,\ell) \subset \Rbar(H^3_{\langle k^\vee_{in},k^\vee_{out}\rangle},s,\ell)$$
to be the closure of the locus $\{\theta_{in} + \theta_{out} = \theta + \pi\}$ in $R^{-1}([0,1))$, where
$$\theta_{in} = \begin{cases}
                    \theta_0^{in} & \text{ if $k_{in} = 0$}\\
                    \arg(z^{in}_1) & \text{ if $k_{in} > 0$,}
\end{cases}     $$
and $\theta_{out}$ is defined analogously. 

\begin{rem}
    Note that the family $\Rbar(H^2_{\langle 0^\vee,0^\vee\rangle},s,\ell)$ coincides with the family $\Rbar(H^2_{\cC\cY},s,\ell)$ defined in \cite[Section 5.6]{relfukii}, except that one of the boundary marked points there is outgoing.
\end{rem}

Let $r_{in}$ denote the action of the generator of $\Z/(s_{in}+1)$ on $\Rbar(A,s,\ell)$ by rotating the labels of the outer boundary marked points, so that $r_{in}^*\theta^{in}_j = \theta^{in}_{j-1}$; and define $r_{out}$ similarly. 
We now inductively define stratifications of $\Rbar(H^{2/3}_{\langle k^\vee_{in},k^\vee_{out}\rangle},s,\ell)$, and simultaneously introduce new families and identify them with their strata, as follows:
\begin{itemize}
    \item The union of codimension-$1$ strata $\{\|z^{in}\| = R\} \subset \Rbar(H^{2/3}_{\langle k^\vee_{in},k^\vee_{out}\rangle},s,\ell)$ is identified with $\Rbar(H^{2/3,S^1}_{\langle k^\wedge_{in},k^\vee_{out}\rangle},s,\ell)$;
    \item $\Rbar(H^{2/3,S^1}_{\langle k^\wedge_{in},k^\vee_{out}\rangle},s,\ell)$ is the union of strata $\{\theta^{in}_j \le \arg(z^{in}_{k^{in}}) \le \theta^{in}_{j+1}\}$, which is identified with $r_{in}^j \cdot \Rbar(H^{2/3}_{\langle k^\wedge_{in},k^\vee_{out}\rangle},s,\ell)$;
    \item The union of codimension-$1$ strata $\{\|z^{out}\| = R\} \subset \Rbar(H^{2/3}_{\langle k^\vee_{in},k^\vee_{out}\rangle},s,\ell)$ is identified with $\Rbar(H^{2/3,S^1}_{\langle k^\vee_{in},k^\wedge_{out}\rangle},s,\ell)$;
    \item  $\Rbar(H^{2/3,S^1}_{\langle k^\vee_{in},k^\wedge_{out}\rangle},s,\ell)$ is the union of strata $\{\theta^{out}_j \le \arg(z^{out}_{k^{out}}) \le \theta^{out}_{j+1}\}$, which is identified with $r_{out}^j \cdot \Rbar(H^{2/3}_{\langle k^\vee_{in},k^\wedge_{out}\rangle},s,\ell)$;
    \item $\Rbar(H^{2/3}_{\langle k_{in}^\wedge,k_{out}^\wedge \rangle}) = \Rbar(H^{2/3}_{\langle k_{in}^\wedge,k_{out}^\vee \rangle }) \cap \Rbar(H^{2/3}_{\langle k_{in}^\vee,k_{out}^\wedge \rangle})$;
    \item $\Rbar(H^{2/3}_{\langle k_{in}^{\diamond_{in}},k_{out}^{\diamond_{out}} \rangle},s,\ell)$ is a stabilization of $\Rbar(H^{2/3}_{\langle k_{in}^{\diamond_{in}},k_{out}^{\diamond_{out}} \rangle},s,0)$, where $s=(\delta(\diamond_{in}),\delta(\diamond_{out}))$, where $\delta(\vee) = 0$ and $\delta(\wedge) = 1$;
    \item The codimension-$1$ stratum $\{R=1\}$ of $\Rbar(H^2_{\langle 0^\vee,0^\vee \rangle},0,0)$ is identified with the attachment of two copies of $\Rbar(\mu,2,0)$ at two boundary marked points;
    \item The codimension-$1$ stratum $\{R=1\}$ of $\Rbar(H^2_{\langle 0^\wedge,0^\vee\rangle},0,0)$ is identified with the self-attachment of $\Rbar(\mu,2,0)$, and similarly for the version with $\wedge$ and $\vee$ swapped;
    \item The codimension-$1$ stratum $\{R=0\}$ of $\Rbar(H^2_{\langle k_{in}^{\diamond_{in}},k_{out}^{\diamond_{out}}\rangle},0,0)$ is identified with $\Rbar(H^{12}_{\langle k_{in}^{\diamond_{in}},k_{out}^{\diamond_{out}}\rangle},0,0)$;
    \item The codimension-$1$ stratum $\{z^{in}_1 = 0\}$ of $\Rbar(H^2_{\langle k_{in}^{\diamond_{in}},k_{out}^{\diamond_{out}}\rangle},s,\ell)$ is identified with $\Rbar(H^3_{\langle (k_{in}-1)^{\diamond_{in}},k_{out}^{\diamond_{out}}\rangle},s,\ell)$;    
    \item The codimension-$1$ stratum $\{z^{out}_1 = 0\}$ of $\Rbar(H^2_{\langle k^{\diamond_{\in}}_{in},k^{\diamond_{out}}_{out}\rangle},s,\ell)$ is identified with $\Rbar(H^3_{\langle k^{\diamond_{in}}_{in},(k_{out}-1)^{\diamond_{out}}\rangle},s,\ell)$.
\end{itemize}

\begin{lem} \label{lem:cardy nontriv} 
    There are maps
    \begin{align*}
        H^{3}_{\langle k_{in}^{\diamond_{in}},k_{out}^{\diamond_{out}}\rangle}: fCC^{\diamond_{in}}_*(\fuk^\big(X,D)) \otimes fCC^{\diamond_{out}}_*(\fuk^\big(X,D)) & \to R^\big\\
        H^{2}_{\langle k_{in}^{\diamond_{in}},k_{out}^{\diamond_{out}}\rangle}: fCC^{\diamond_{in}}_*(\fuk^\big(X,D)) \otimes fCC^{\diamond_{out}}_*(\fuk^\big(X,D)) & \to R^\big
    \end{align*}
    with 
    \begin{align}
\label{eq:sigma H3}        \sigma(H^3_{\langle k_{in}^{\diamond_{in}},k_{out}^{\diamond_{out}} \rangle}) & = \sigma_{in}(\cO\cC)\sigma_{out}(\cO\cC)\sigma(-2n-2k_{in}-2k_{out}-2)\\
        \label{eq:sigma H2}\sigma(H^2_{\langle k_{in}^{\diamond_{in}},k_{out}^{\diamond_{out}} \rangle}) &= \sigma(\partial)^\vee \sigma(2)\sigma(H^3_{\langle k_{in}^{\diamond_{in}},k_{out}^{\diamond_{out}} \rangle});
    \end{align}
    and these satisfy
    \begin{multline}\label{eq:H12 vee vee}
        H^{12}_{\langle k^\vee_{in},k^\vee_{out}\rangle}(\alpha,\beta) - (-1)^\signn\delta_{k_{in},0}\delta_{k_{out},0} \langle \alpha,\beta \rangle_{Muk,\vee\vee} = \partial(H^2_{\langle k^\vee_{in},k^\vee_{out}\rangle}(\alpha,\beta)) \\+ H^2_{\langle k^\vee_{in},k^\vee_{out}\rangle}(b\alpha,\beta) + H^2_{\langle k^\vee_{in},k^\vee_{out}\rangle}(\alpha,b\beta) + H^2_{\langle (k_{in}-1)^\wedge,k^\vee_{out}\rangle}(B\alpha,\beta) \\+ H^2_{\langle k^\vee_{in},(k_{out}-1)^\wedge\rangle}(\alpha,B\beta) - H^3_{\langle (k_{in}-1)^\vee,k^\vee_{out}\rangle}(\alpha,\beta) - H^3_{\langle k^\vee_{in},(k_{out}-1)^\vee\rangle}(\alpha,\beta)
    \end{multline}
    (where $\delta_{k,0}$ is a Kronecker delta: $1$ if $k=0$, $0$ otherwise; and by definition $H^3_{\langle k_{in},k_{out}\rangle}$ is $0$ unless $k_{in}\ge 0$ and $k_{out} \ge 0$);
    \begin{multline}
        H^{12}_{\langle k^\wedge_{in},k^\vee_{out}\rangle}(\alpha,\beta) - (-1)^\signn\delta_{k_{in},0}\delta_{k_{out},0} \langle \alpha,\beta \rangle_{Muk,\wedge\vee} = \partial(H^2_{\langle k^\wedge_{in},k^\vee_{out}\rangle}(\alpha,\beta)) \\+ H^2_{\langle k^\wedge_{in},k^\vee_{out}\rangle}(b_{\wedge\wedge}\alpha,\beta) + H^2_{\langle k^\vee_{in},k^\vee_{out}\rangle}(b_{\wedge\vee}\alpha,\beta) + H^2_{\langle k^\wedge_{in},k^\vee_{out}\rangle}(\alpha,b\beta)  \\+ H^2_{\langle k^\vee_{in},(k_{out}-1)^\wedge\rangle}(\alpha,B\beta) - H^3_{\langle (k_{in}-1)^\wedge,k_{out}^\vee\rangle}(\alpha,\beta) - H^3_{\langle k_{in}^\wedge,(k_{out}-1)^\vee\rangle}(\alpha,\beta);
    \end{multline}
    the analogous version with $\vee$ and $\wedge$ swapped; and
    \begin{multline}
        H^{12}_{\langle k^\wedge_{in},k^\wedge_{out}\rangle}(\alpha,\beta) = \partial(H^2_{\langle k^\wedge_{in},k^\wedge_{out}\rangle}(\alpha,\beta)) \\+ H^2_{\langle k^\wedge_{in},k^\wedge_{out}\rangle}(b_{\wedge\wedge}\alpha,\beta) + H^2_{\langle k^\vee_{in},k^\wedge_{out}\rangle}(b_{\wedge\vee}\alpha,\beta) + H^2_{\langle k^\wedge_{in},k^\wedge_{out}\rangle}(\alpha,b_{\wedge\wedge}\beta)  \\+ H^2_{\langle k_{in}^\wedge,k_{out}^\vee \rangle}(\alpha,b_{\wedge\vee}(\beta) - H^3_{\langle (k_{in}-1)^\wedge,k^\wedge_{out}\rangle}(\alpha,\beta) - H^3_{\langle k^\wedge_{in},(k_{out}-1)^\wedge\rangle}(\alpha,\beta).
    \end{multline}
\end{lem}
\begin{proof}
    Follows by applying \cite[Lemma 4.18]{relfukii} to the families $H^2_{\langle \rangle}$. Note that, while we define operations from the families $H^3_{\langle \rangle}$, we do not apply \cite[Lemma 4.18]{relfukii} to them.
\end{proof}

\begin{cor}\label{cor:cardy 2}
    There is a $u$-sesquilinear map
    \begin{align*}
        H^2_{\langle\rangle}: \sigma(\partial)^\vee\sigma_{in}(\cO\cC)\sigma_{out}(\cO\cC)fCC_*^-(\fuk^\big) \otimes fCC_*^-(\fuk^\big) & \to R^\big[[u]],
    \end{align*}
    satisfying
    \begin{multline}
H^{12}_{\langle\rangle}(\alpha,\beta) - (-1)^\signn\langle \alpha,\beta \rangle_{res} = \partial(H^2_{\langle\rangle}(\alpha,\beta)) \\+ H^2_{\langle\rangle}((b+uB)(\alpha),\beta) + H^2_{\langle\rangle}(\alpha,(b+uB)(\beta)).
    \end{multline}
\end{cor}
\begin{proof}
    We define $H^2_{\langle\rangle}$ by the analogue of \eqref{eq:put the us}; then the result follows from Lemma \ref{lem:cardy nontriv}. 
    Note that the terms $H^3_{\langle \rangle}$ cancel in the sum.
\end{proof}

\begin{proof}[Proofs of Theorems \ref{thm:mukai_int} and \ref{thm:oc_hres_bc}]
    Combining Corollaries \ref{cor:cardy 1} and \ref{cor:cardy 2}, we obtain a homotopy
    $$H_{\langle \rangle} = H^1_{\langle \rangle} + H^2_{\langle \rangle}$$
    such that
    $$\langle \cO\cC^{-,\big}(\alpha),\cO\cC^{-,\big}(\beta) \rangle - (-1)^\signn\langle \alpha,\beta \rangle_{res} = \partial ( H)(\alpha,\beta).$$
    Tensoring with $\Lambda^\big_{\ge 0}$, then passing to the category of small bounding cochains using Lemma \ref{lem:C_Cbc_pairings}, tensoring with $\Lambda^\big$, then passing to cohomology, we obtain Theorem \ref{thm:oc_hres_bc}. 
    If we quotient by the bulk variables before taking cohomology, then we obtain Theorem \ref{thm:mukai_int}.
\end{proof}

\section{Cyclic open--closed map respects connections}\label{sec:oc conn}

In this section we prove that the cyclic open--closed map respects connections, proving Theorem \ref{thm:oc_conn_bc}.

\subsection{Dubrovin--Givental connection}\label{sec:pdvq}

Recall that we define $QC^*(X;R^\big) := CM^*(f,g) \otimes R^\big$, which comes equipped with the Morse differential $\partial_{QC}$, and the big quantum cup product $\star$.  

We now define chain maps
$$PD(V_q) \star (-): QC^*(X;R) \to QC^*(X;R)$$
for each $q \in Q$, and arrange that the induced maps on cohomology coincide with the quantum cup product with $PD(V_q)$. 
We emphasize that these chain-level maps $PD(V_q) \star$ are new operations, which were not defined in \cite{relfukii}: for example, there is no natural chain-level representative of $PD(V_q)$. 

We define $\Rbar(PD(V_*) \star,0,0)$ to be the moduli space of spheres with one non-symmetric stabilizing marked point $p_*^{int}$, one incoming bulk marked point $p_1^{int}$, and one outgoing bulk marked point $p_0^{int}$. 
We define $\Rbar(PD(V_*)\star,\ell)$ to be its stabilization, adding $\ell_\bulk$ incoming bulk marked points and $\ell_\stab$ symmetric stabilizing marked points. 
We define 
$$PD(V_q) \star (-):= \sum_{\ell} (PD(V_*) \star,\ell)_{q}.$$
It follows from \cite[Lemma 4.18]{relfukii} that it is a chain map.

\begin{defn}
    Let 
    $$D_R^\big: R^\big \to \Omega^\big_R = R^\big \otimes (\Z^Q \oplus CM^{2-*}(f,g))$$ 
    be the derivation introduced in Example \ref{eg:Omega R big}. 
    The \emph{Dubrovin--Givental connection}
    $$u\nabla^{DG}_R: QC^*(X;R^\big)[[u]] \to \Omega_R^\big \otimes QC^*(X;R^\big)[[u]]$$
    is defined by
    $$u\nabla^{DG}_R (\alpha) := uD_R^\big(\alpha) - \sum_{q \in Q} d\log \nov_q \otimes PD(V_q) \star \alpha - \sum_i d\novb_i \otimes \beta_i \star \alpha,$$
    extended $u$-linearly.
\end{defn}

Note that it follows from the fact that $\star$ satisfies the Leibniz rule (by \cite[Lemma 5.3]{relfukii}), and $PD(V_q)\star(-)$ is a chain map, that $\nabla^{DG}_R$ is indeed a connection in the sense of Definition \ref{def:conn}.

\subsection{Cyclic open--closed map respects connections}

We define families $\Rbar(\partial_*\cO\cC,k^\diamond,s,\ell)$ (respectively, $\Rbar(\partial_*\mu,s,\ell)$) to be the same as $\Rbar(\cO\cC,k^\diamond,s,\ell)$ (respectively, $\Rbar(\mu,s,\ell)$), except that one of the stabilizing marked points $p_*^{int}$ is designated as non-symmetric whenever it does not lie on a sphere bubble; and symmetric whenever it lies on a sphere bubble.
We choose perturbation data for this family by pullback under the natural map $\Rbar(\partial_*\cO\cC,k^\diamond,s,\ell) \to \Rbar(\cO\cC,k^\diamond,s,\ell)$ (respectively, $\Rbar(\partial_*\mu,s,\ell) \to \Rbar(\mu,s,\ell)$); we observe that it satisfies all of the required conditions for a choice of perturbation data.

\begin{lem}
    For each $q \in Q$, we have 
    \begin{align*}
        D^q_R(\cO\cC^{-,\big}) &= \sum_{s,\ell,k^\diamond}(\partial_*\cO\cC,k^\diamond,s,\ell)^{\big}_{\emptyset,\bL,q} \cdot u^k \cdot d\log\nov_q \qquad\text{and}\\
        D^q_R(\mu) &= \sum_{s,\ell}(\partial_*\mu,s,\ell)^{\big}_{\emptyset,\bL,q} \cdot d\log\nov_q.
    \end{align*}
\end{lem}
\begin{proof}
    This follows because each curve $u$ contributing $\nov^u \cdot \beta$ to the LHS, where $u \in \NE$ and $\beta$ is a monomial in the bulk variables,  contributes $(u \cdot V_q) \cdot \nov^u \cdot \beta$ to the RHS. 
    This is because there are $u \cdot V_q$ choices of stabilizing point to designate as $p_*^{in}$, among those whose tangency vector constrains them to lie on $V_q$. 
    It now suffices to observe that $D_q^R(\nov^u \cdot \beta) = (u \cdot V_q) \cdot \nov^u \cdot \beta \cdot d\log \nov_q$, by definition.
\end{proof}

We define families $\Rbar(H^{1}_\nabla,k^\diamond,s,\ell)$, where $k^\diamond \in K$, $s \in \mathbb{N}_0$, and $\ell=(\ell_{\bulk,\cO\cC},\ell_{\stab,\cO\cC},\ell_{\bulk,\nabla},\ell_{\stab,\nabla})$, by attaching the outgoing bulk marked point of $\Rbar(\cO\cC,k^\diamond,s,(\ell_{\bulk,\cO\cC},\ell_{\stab,\cO\cC}))$ to the first incoming bulk marked point of $\Rbar(PD(V_*)\star,(\ell_{\bulk,\nabla},\ell_{\stab,\nabla}))$, with length parameter $[0,\infty]$. 
We denote the boundary stratum where the length parameter is $0$, by $\Rbar(H^{12}_\nabla,k^\diamond,s,\ell)$.

\begin{lem}\label{lem:conn 1}
    For each $q \in Q$, there are maps
    \begin{align*}
        H^1_{\nabla_q}:\sigma(\partial)^\vee \sigma(\cO\cC) fCC_*^-(\fuk^\big(X,D)) & \to QC^*(X;R^\big)[[u]],\\
        H^{12}_{\nabla_q}:\sigma(\cO\cC)fCC_*^-(\fuk^\big(X,D)) &\to QC^*(X;R^\big)[[u]]
    \end{align*}
    satisfying
    \begin{align}\label{eq:H1 nab}
        PD(V_q) \star \cO\cC^{-,\big}(-) - H^{12}_{\nabla_q} &= \partial(H^1_{\nabla_q})
    \end{align}
    where $\sigma(\partial) = \sigma(b) = \sigma(2)\sigma(B)$.
\end{lem}
\begin{proof}
    We define maps
    \begin{align*}
        H^1_{\nabla_q,k^\diamond}:\sigma(\partial)^\vee\sigma(\cO\cC)\sigma(-2k) fCC_*^\diamond(\fuk^\big(X,D)) &\to QC^*(X;R^\big),\\
        H^{12}_{\nabla_q,k^\diamond}:\sigma(\cO\cC)\sigma(-2k) fCC_*^\diamond(\fuk^\big(X,D)) &\to QC^*(X;R^\big),
    \end{align*}
    by
    $$H^{1/12}_{\nabla_q,k^\diamond} := \sum_{s,\ell} (H^{1/12}_{\nabla},k^\diamond,s,\ell)_{\emptyset,\bL,q}.$$
    By applying \cite[Lemma 4.18]{relfukii} to the families $\Rbar(H^1_\nabla,k^\diamond,s,\ell)$, we find that they satisfy
    \begin{align*}
        PD(V_q) \star \cO\cC^\vee_k - H^{12}_{\nabla_q,k^\vee} = H^1_{\nabla_q,k^\vee} \circ b_{\vee\vee} + H^1_{\nabla_q,(k-1)^\wedge} \circ B\\
        PD(V_q) \star \cO\cC^\wedge_k - H^{12}_{\nabla_q,k^\wedge} = H^1_{\nabla_q,k^\wedge} \circ b_{\wedge\wedge} + H^1_{\nabla_q,k^\vee} \circ b_{\wedge\vee}.
    \end{align*}
    The result then follows by summing over powers of $u$ as in Section \ref{sec:oc pair}.
\end{proof}

We define the family $\Rbar(\cO\cC_*,0,0)$ to be the moduli space of discs with one incoming boundary marked point $p_0^\partial$, one outgoing bulk marked point $p_0^{int}$, and one non-symmetric stabilizing marked point $p_*^{int}$. 
By parametrizing the disc as the unit disc in $\C$ with $p_0^{int}$ at the origin and $p_0^\partial$ at $1 \in S^1$, we obtain a map from the boundary of any disc in the family to $S^1$. 
We define an analytic function $t:\Rbar(\cO\cC_*,0,0)$ to $[0,1]$, so that the disc can be parametrized as the unit disc in $\C$ with $p_0^{int}$ at $-t$ and $p_*$ at $+t$. 
For $t \in (0,1]$, we define $\theta_* \in S^1$ to be the point which gets sent to $-i$, under the above parametrization. The function $\theta_*$ does not extend over $t=0$.

We define $\Rbar(\cO\cC_*,s,\ell)$ to be the stabilization of $\Rbar(\cO\cC_*,0,0)$. 
There are analytic functions $\theta_j \in S^1$ defined on this family, which record the position of the $j$th boundary marked point, for $0 \le j \le s$, with $\theta_0 = 1$; and there is also an analytic function $t \in [0,1]$, and an analytic function $\theta_* \in S^1$ defined on $t^{-1}((0,1])$, both of which are pulled back from $\Rbar(\cO\cC_*,0,0)$ under the stabilization map. 

We define the families $\Rbar(H^2_\nabla,k^\vee,s,\ell)$ to be the closure of the locus $\{\theta_0=\theta_*\}$ in $\Rbar(\cO\cC_*,s,\ell) \times D_k$; and $\Rbar(H^2_\nabla,k^\wedge,s,\ell)$ to be the closure of the locus $\{\theta_0 \le \theta_* \le \theta_1\}$. 
We inductively define stratifications of $\Rbar(H^2_\nabla,k^\diamond,s,\ell)$, and simultaneously introduce new families and identify them with these strata, as follows:
\begin{itemize}
    \item the union of codimension-$1$ strata $z \in \partial D_k$ is identified with $\Rbar(H^{23}_\nabla,k^\diamond,s,\ell)$;
    \item the codimension-$1$ stratum $\theta_0= \theta_*$ of $\Rbar(H^2_\nabla,k^\wedge,s,\ell)$ is identified with $\Rbar(H^2_\nabla,k^\vee,s,\ell)$;
    \item the codimension-$1$ stratum $\theta_* = \theta_1$ of $\Rbar(H^2_\nabla,k^\wedge,s,\ell)$ is identified with $r\cdot \Rbar(H^2_\nabla,k^\vee,s,\ell)$
    \item the codimension-$1$ stratum $\{t=0\}$ is identified with $\Rbar(H^{12}_\nabla,k^\diamond,s,\ell)$;
    \item the codimension-$1$ stratum $\{t=1\}$ is identified with:
    \begin{itemize}
        \item if $\diamond = \vee$, a boundary attachment of $\Rbar(\partial_*\mu,s,\ell)$, $\Rbar(\mu,s,\ell)$, and $\Rbar(\cO\cC,k^\vee,s,\ell)$;
        \item if $\diamond = \wedge$, a boundary attachment of $\Rbar(\partial_*\mu,s,\ell)$ with $\Rbar(\cO\cC,k^\vee,s,\ell)$.
    \end{itemize}
\end{itemize}

\begin{lem}\label{lem:conn 2}
    For each $q \in Q$, there are maps
    \begin{align*}
        H^2_{\nabla_q}:\sigma(\partial)^\vee \sigma(\cO\cC) fCC_*^-(\fuk^\big(X,D)) & \to QC^*(X;R^\big)[[u]],\\
        H^{23}_{\nabla_q}:\sigma(\cO\cC)fCC_*^-(\fuk^\big(X,D)) &\to QC^*(X;R^\big)[[u]]
    \end{align*}
    satisfying
    \begin{align}\label{eq:H2 nab}
         H^{12}_{\nabla_q} + H^{23}_{\nabla_q}- \cO\cC^{-,\big}(b^{1|1}(D_q^R(\mu),-)) &= \partial \circ H^2_{\nabla_q} + H^2_{\nabla_q} \circ b.
    \end{align}
\end{lem}
\begin{proof}
    The proof follows that of Lemma \ref{lem:conn 1}, by applying \cite[Lemma 4.18]{relfukii} to the families $\Rbar(H^2_\nabla,k^\diamond,s,\ell)$, then summing over powers of $u$.
\end{proof}

We subdivide the family $\{z \in \partial D_k\} \subset \Rbar(\cO\cC_*,s,\ell) \times D_k$ into a union of the following families:
\begin{align*} 
&\{\theta_j \le \arg(z_k) \le \theta_{j+1} \le \theta_k \le \theta_* \le \theta_{k+1} \le \theta_j\}; \\
&\{\theta_j \le \arg(z_k) \le \theta_* \le \theta_{j+1}\}; \\
&\{\theta_j \le \theta_* \le \arg(z_k)\le\theta_{j+1}\},
\end{align*}
for all $j$ and $k$, where ``$\le$'' refers to the cyclic order on $S^1$. 
When we refer to a family such as ``$\{\arg(z_k)=\theta_*\}$'', we are implicitly referring to the corresponding union of strata (when we define an operation corresponding to a union of strata, we mean the sum of the operations associated to the pieces). 
We define the family $\Rbar(H^3_\nabla,k^\vee,s,\ell)$ to be $\{\theta_0 \le \arg(z_k) \le \theta_*\}$. 
We identify the codimension-$1$ boundary strata as follows:
\begin{itemize}
    \item $\{\theta_0 = \theta_*\}$ is identified with $\Rbar(H^{23}_\nabla,k^\vee,s,\ell)$; 
    \item $\{\theta_0=\arg(z_k)\}$ is identified with $\Rbar(H^{34}_\nabla,k^\vee,s,\ell)$;
    \item $\{\theta_j \le \theta_* = \arg(z_k) \le \theta_{j+1}\}$ is identified with $r^j\cdot \Rbar(H^2_\nabla,(k-1)^\wedge,s,\ell)$;
    \item $\{t=1\}$ is identified with a boundary attachment of $\Rbar(\cO\cC,k^\wedge,s,\ell)$ with $\Rbar(D_q^R(\mu),s,\ell)$.
\end{itemize}

We define the family $\Rbar(H^{34}_\nabla,k^\wedge,s,\ell)$ to be the difference $\{\theta_0 \le \arg(z_k) \le \theta_1\}$. 
We identify the codimension-$1$ boundary strata as follows:
\begin{itemize}
    \item $\{\theta_0 = \arg(z_k)\}$ is identified with $\Rbar(H^{34}_\nabla,k^\vee,s,\ell)$;
    \item $\{\arg(z_k) = \theta_1\}$ is identified with $r \cdot \Rbar(H^{34}_\nabla,k^\vee,s,\ell)$.
\end{itemize}

\begin{lem}\label{lem:conn 3}
    For each $q \in Q$, there are maps
    \begin{align*}
        H^3_{\nabla_q}:\sigma(\partial)^\vee \sigma(\cO\cC) fCC_*^-(\fuk^\big(X,D)) & \to QC^*(X;R^\big)[[u]],\\
        H^{34}_{\nabla_q}:\sigma(\cO\cC)fCC_*^-(\fuk^\big(X,D)) &\to QC^*(X;R^\big)[[u]]
    \end{align*}
    satisfying
    \begin{align}\label{eq:H3 nab}
         -H^{23}_{\nabla_q} + H^{34}_{\nabla_q} - \cO\cC^{-,\big}(B^{1|1}(D_q^R(\mu),-)) &= \partial(H^3_{\nabla_q}) + H^2_{\nabla_q} \circ uB.
    \end{align}
\end{lem}
\begin{proof}
    We define 
    \begin{align*}
        H^3_{\nabla_q,k^\vee}:\sigma(\partial)^\vee\sigma(\cO\cC)\sigma(-2k) fCC_*^\vee(\fuk^\big(X,D)) &\to QC^*(X;R^\big),\\
        H^{34}_{\nabla_q,k^\diamond}:\sigma(\cO\cC)\sigma(-2k) fCC_*^\diamond(\fuk^\big(X,D)) &\to QC^*(X;R^\big),
    \end{align*}
    as in the proof of Lemma \ref{lem:conn 1}; and we define $H^3_{\nabla_q,k^\wedge}$ to be $0$ (note that we have not defined families $\Rbar(H^3_\nabla,k^\wedge,s,\ell)$; we may take these to be the empty family). 
    Applying \cite[Lemma 4.18]{relfukii} to the families $\Rbar(H^3_\nabla,k^\vee,s,\ell)$, we obtain
    \begin{align}\label{eq:H23 H34}
        H^{23}_{\nabla_q,k^\vee} - H^{34}_{\nabla_q,k^\vee} + \cO\cC^\wedge_{k-1} \circ B^{1|1}&=  H^3_{\nabla_q,k^\vee} \circ b_{\vee\vee} + H^2_{\nabla_q,(k-1)^\wedge} \circ B.
    \end{align}
    We now observe that
    \begin{align}\label{eq:H3 wedge vee}
        H^3_{\nabla_q,k^\vee} \circ b_{\wedge\vee} &= H^{23}_{\nabla_q,k^\wedge} - H^{34}_{\nabla_q,k^\wedge},
    \end{align}
    which follows from the following identity of unions of strata:
    $$r \cdot \Rbar(H^3_\nabla,k^\vee,s,\ell) - \Rbar(H^3_\nabla,k^\vee,s,\ell) = \Rbar(H^{34}_\nabla,k^\wedge,s,\ell) - \Rbar(H^{23}_\nabla,k^\wedge,s,\ell).$$
    The result now follows by summing over powers of $u$.
\end{proof}

We define the family $\Rbar(H^4_\nabla,k^\diamond,s,\ell)$ to be $\Rbar(\partial_*\cO\cC,k^\diamond,s,\ell) \times [0,1]_\rho$, where the codimension-$1$ stratum $\{\rho=1\}$ is identified with $\Rbar(\partial_*\cO\cC,k^\diamond,s,\ell)$ (note that in particular, the stabilizing marked point $p_*^{int}$ becomes symmetric along sphere bubbles inside this stratum), and the codimension-$1$ stratum $\{\rho=0\}$ is identified with $\Rbar(H^{34}_\nabla,k^\diamond,s,\ell)$. 

\begin{lem}\label{lem:conn 4}
    For each $q \in Q$, there are maps
    \begin{align*}
        H^4_{\nabla_q}:\sigma(\partial)^\vee \sigma(\cO\cC) fCC_*^-(\fuk^\big(X,D)) & \to QC^*(X;R^\big)[[u]]
    \end{align*}
    satisfying
    \begin{align}\label{eq:H4 nab}
         D_q^R(\cO\cC^{-,\big}) - H^{34}_{\nabla_q} &= \partial(H^4_{\nabla_q}).
    \end{align}
\end{lem}
\begin{proof}
    Follows by applying \cite[Lemma 4.18]{relfukii} to the families $\Rbar(H^4_\nabla,k^\diamond,s,\ell)$, then summing over powers of $u$ as in the previous section. 
\end{proof}

\begin{rem}    
Because any sphere bubble in the moduli spaces defining $\cO\cC^{-,\big}$ is not $f^{sym}$-stable, we are forced to make $p_*^{int}$ become symmetric along the stratum $\{\rho=1\}$, in order for Lemma \ref{lem:conn 4} to hold.   
On the other hand, we \emph{may} make $p_*^{int}$ symmetric along this locus, because the locus inside $\Rbar(\partial_*\cO\cC,k^\diamond,s,\ell)$ where the point $p_*^{int}$ is symmetric and lies on a sphere bubble, has codimension $2$. 
In particular, $\Rbar(\partial_*\cO\cC,k^\diamond,s,\ell)$ is $f^{sym}$-stable, so we may define operations from it, and we may apply \cite[Lemma 4.18]{relfukii} to $\Rbar(H^4_\nabla,k^\diamond,s,\ell)$. 
On the other hand, we can't allow $p_*^{int}$ to be symmetric in all of the moduli spaces we consider: sitting inside the stratum $\{\rho = 0\} = \Rbar(H^{34}_\nabla,k^\diamond,s,\ell)$, we have the stratum $\{\theta_* = \theta_0\}$, which is identified with $\Rbar(H^2_\nabla,(k-1)^\vee,s,\ell)$; and inside that stratum we have the stratum $\{t=1\}$, which is identified with $\Rbar(H^{12}_\nabla,(k-1)^\vee,s,\ell)$. 
We need to define an operation $H^{12}_\nabla$ from this stratum, so it needs to be $f^{sym}$-stable; this forces $p_*^{int}$ to be non-symmetric along this stratum, for if $p_*^{int}$ were symmetric then by convention the almost-complex structure would need to be constant along the sphere bubble, and in that case we can only guarantee regularity for simple spheres.
\end{rem}

\begin{cor}
    \label{cor:oc_conn}
    For each $q \in Q$, there is a map of $R^\big[[u]]$-modules,
    $$H_{\nabla_q}: \sigma(\partial)^\vee\sigma(\cO\cC) fCC_*^-(\fuk^\big(X,D)) \to \Omega^q_R \otimes QC^*(X;R)[[u]],$$
    such that
    $$ \nabla^{DG}_R \circ \cO\cC^{-,\big} - (\id \otimes \cO\cC^{-,\big}) \circ \nabla^{GGM}  = \partial ( H_{\nabla_q}).$$
    In this formula, the filtered connections $\tilde \nabla$ on the morphism spaces of $\fuk^\big(X,D)$, which are used in the chain-level definition of $\nabla^{GGM}$, are induced by the natural choice of basis for each morphism space (i.e., the connections are characterized by the fact that they make morphisms in $\fuk(X \setminus D) \otimes 1 \subset \fuk^\big(X,D)$, $\tilde \nabla$-constant).
\end{cor}
\begin{proof}
    We define $H_{\nabla_q} := H^1_{\nabla_q} + H^2_{\nabla_q} + H^3_{\nabla_q} + H^4_{\nabla_q}$. 
    Then the result follows by combining Lemmas \ref{lem:conn 1}--\ref{lem:conn 4}, together with the observation that
    $$D_q^R \circ \cO\cC^{-,\big} - \cO\cC^{-,\big} \circ D_q^R = D_q^R(\cO\cC^{-,\big}).$$
\end{proof}

\begin{prop}\label{prop:h nabla i}
    For each basis element $\beta_i$ of $CM^{2-*}(f,g)$, there is a map of $R^\big[[u]]$-modules,
    $$H_{\nabla_i}: \sigma(\partial)^\vee\sigma(\cO\cC) fCC_*^-(\fuk^\big(X,D)) \to \Omega^i_R \otimes QC^*(X;R)[[u]],$$
    such that
    $$ \nabla^{DG}_R \circ \cO\cC^{-,\big} - (\id \otimes \cO\cC^{-,\big}) \circ \nabla^{GGM}  = \partial ( H_{\nabla_i}).$$
\end{prop}
\begin{proof}
    The proof is identical to that of Corollary \ref{cor:oc_conn}, except that the point $p_*^{int}$ is designated as an incoming bulk marked point labelled by the $i$th critical point of the Morse function $f$, rather than a stabilizing marked point labelled by $q$, throughout the argument. 
\end{proof}

\begin{thm}
    \label{thm:oc_conn}
    There exists a map of $R^\big[[u]]$-modules,
    $$H_\nabla: \sigma(\partial)^\vee\sigma(\cO\cC) fCC_*^-(\fuk^\big(X,D)) \to \Omega^\big_R \otimes u^{-1} QC^*(X;R)[[u]],$$
    such that
    $$ \nabla^{DG}_R \circ \cO\cC^{-,\big} - (\id \otimes \cO\cC^{-,\big}) \circ \nabla^{GGM} = \partial ( H_\nabla )$$
    (where the connections $\tilde \nabla$ involved in the definition of $\nabla^{GGM}$ are chosen as in Corollary \ref{cor:oc_conn}).
\end{thm}
\begin{proof}
    We define $H_\nabla = \bigoplus_{q \in Q} H_{\nabla_q} \oplus \bigoplus_i H_{\nabla_i}$, then the result follows from Corollary \ref{cor:oc_conn} and Proposition \ref{prop:h nabla i}.
\end{proof}

\begin{proof}[Proof of Theorem \ref{thm:oc_conn_bc}]
    Let $(f^*,Df): (R^\big,\Omega^\big_R) \to (\Lambda^\big_{\ge 0},\Omega^\big_{\Lambda,\ge 0})$ be the morphism of coefficient rings with derivations from Lemma \ref{lem:RLamder}; then we may pull back $(R^\big,D^\big_R)$-connections along $(f^*,Df)$ by Lemma \ref{lem:pullback conn}. 
    It follows from Theorem \ref{thm:oc_conn} that
    \begin{multline}
        f^*\nabla^{DG}_R \circ (\cO\cC^{-,\big} \otimes_R \Lambda_{\ge 0}) - (\cO\cC^{-,\big} \otimes_R \Lambda_{\ge 0}) \circ f^*\nabla^{GGM} =\\ \partial (H_\nabla \otimes_R \Lambda_{\ge 0}) .
    \end{multline}
    By Lemma \ref{lem:Rlamderconn}, we have $f^*\nabla^{DG}_R = \nabla^{DG}_\Lambda$; and by the naturality of the Getzler--Gauss--Manin connection, we have that $f^*\nabla^{GGM}$ is the Getzler--Gauss--Manin connection for $\fuk^\big(X,D;\Lambda^\big_{\ge 0})$. 

    Now recall that the next step in defining $\cO\cC^{-,\big}_\Lambda$ is to pre-compose $\cO\cC^{-,\big} \otimes_R \Lambda^\big_{\ge 0}$ with the map 
    $$F_*:fCC_*^-(\fuk^\big(X,D;\Lambda^\big_{\ge 0})^\sbc) \to fCC_*^-(\fuk^\big(X,D;\Lambda^\big_{\ge 0}))$$ from Lemma \ref{lem:HC_C_Cbc}; this map respects connections up to homotopy, by Lemma \ref{lem:GGM_C_Cbc}.
    The final step is to base change along the inclusion $\Lambda^\big_{\ge 0} \to \Lambda^\big$, which clearly respects connections. 
    Putting it all together, we obtain that $\cO\cC^{-,\big}_\Lambda$ respects connections up to homotopy, and in particular, respects connections on the level of homology.
\end{proof}

\appendix

\section{Signs}\label{sec:signs}

We justify selected signs for the relations proved in this paper; the other sign checks are straightforward variations on the ones explained here.

\subsection{Cyclic open--closed map}\label{sec:cyc oc signs}

We use the complex orientation on $D_k$ to get an isomorphism
$$\S(\cO\cC^\vee_k) = \sigma(\cO\cC)\sigma(-2k),$$
and hence define $\cO\cC^\vee_k$. 
As $\cO\cC^\wedge_{k-1}$ forms part of the codimension-$1$ boundary of $\cO\cC^\vee_{k}$, we have the boundary orientation isomorphism
$$\S(\cO\cC^\wedge_{k-1}) = \sigma(\partial)\S(\cO\cC^\vee_k) = \sigma(\cO\cC)\sigma(-2k);$$
by identifying $\sigma(\partial) \sigma(e^+) = \sigma(2)$, we obtain an identification
$$\S(\cO\cC^\wedge_{k-1}) = \sigma(\cO\cC)\sigma(-2(k-1)) \sigma(e^+)^\vee,$$
which allows us to define $\cO\cC^\wedge_{k-1}$.

Together, these identifications suffice to define the signs in the operations $\cO\cC^\vee_k$ and $\cO\cC^\wedge_k$; putting them together with the natural isomorphism $\sigma(B) = \sigma(-2)\sigma(\partial)$ suffices to verify the signs in \eqref{eq:OC check rel}. 
The signs at codimension-1 strata corresponding to disc bubbles are treated by \cite[Lemma C.1]{relfukii}. 
Note that a key point is that, at the boundary component $\cR(\cO\cC,(k-1)^\wedge,s,\ell)$ of $\cR(\cO\cC,k^\vee,s,\ell)$, if we define $z_k = r_k\exp(i\theta_k)$, then we may identify $\sigma(r_k)$ with $\sigma(\partial)$ and $\sigma(\theta_k)$ with $\sigma(e^+)$; thus the identification $\sigma(\partial)\sigma(e^+) = \sigma(2)$ corresponds to the complex orientation of the complex $z_k$-plane. 

To verify the signs in \eqref{eq:OC hat rel}, we need to compare the boundary orientation at the codimension-one boundary components $r \cdot \cO\cC^\vee_k$ and $\cO\cC^\vee_k$ of $\cO\cC^\wedge_k$, with the complex orientation. 
This is equivalent to comparing the boundary orientation of the boundary points $\{\theta_1\}$ and $\{\theta_0\}$ of the interval $\{z \in S^1: \theta_0 \le \arg(z) \le \theta_1\}$, where $S^1$ is equipped with the standard orientation (as this is the one induced on it as the boundary of the unit disc in $\C$). 
The orientations agree at $\{\theta_1\}$, and disagree at $\{\theta_0\}$; this agrees with the signs of the two terms in \eqref{eq:b hat}, which suffices to verify the signs in \eqref{eq:OC hat rel}.

\subsection{Cyclic open--closed map respects pairings}

Recalling that $\cR(A,0,0) \cong \{(R,\theta):R \in (0,1),\theta \in S^1\}$, we obtain an identification $\sigma(\cR(A,0,0)) = \sigma(R)\sigma(\theta)$, where $\sigma(R) = \sigma((0,1))$ and $\sigma(\theta) = \sigma(S^1)$ are given their usual orientations. 

We have an identification $\sigma(D_{k_{in},k_{out}}) = \sigma(2(k_{in}+k_{out})) \sigma(R)$, by equipping $\C^{k_{in}} \times \C^{k_{out}}$ with the complex orientation. 

These give rise to identifications 
\begin{equation}\label{eq:RH3 or}
    \sigma(\cR(H^3_{\langle k_{in}^\vee,k_{out}^\vee\rangle},(1,1),(0,0))) = \sigma(R)\sigma(\theta)\sigma(2k_{in}+2k_{out}),
\end{equation}
and hence
$$\S(H^3_{\langle k_{in}^\vee,k_{out}^\vee}\rangle ) = \sigma(R)^\vee\sigma(\theta)^\vee \sigma(-2n-2k_{in}-2k_{out}) \sigma_{in}(\cO\cC)\sigma_{out}(\cO\cC),$$
where $\sigma_{in/out}(\cO\cC) = \sigma(2n)\sigma(B_{in/out})^\vee$. 

We have
$$\sigma(\cR(H^2))\sigma(N_{H^2/H^3})=\sigma(\cR(H^3)),$$
where $N_{H^3/H^2}$ denotes the normal bundle to $\cR(H^2) = \{\theta_{in}+\theta_{out} = \theta+\pi\}$ inside $\cR(H^3)$. 
We orient $\sigma(N_{H^3/H^2})$ in the direction of increasing $\theta_{in} + \theta_{out} - \theta$. 

This gives rise to an identification 
$$\sigma(\cR(H^2_{\langle k_{in}^\vee,k_{out}^\vee\rangle},(1,1),(0,0))) = \sigma(R)\sigma(2k_{in}+2k_{out}).$$

Identifying $\sigma(R) = \sigma(\partial)$, we obtain the identification
$$\S(H^2_{k_{in}^\vee,k_{out}^\vee}) = \sigma(\partial)^\vee \sigma(-2n-2k_{in}-2k_{out}) \sigma_{in}(\cO\cC)\sigma_{out}(\cO\cC) = \sigma(H^2_{\langle k_{in}^\vee,k_{out}^\vee\rangle})$$
(cf. \eqref{eq:sigma H2}), which allows us to define $H^2_{k_{in}^\vee,k_{out}^\vee}$. 
On the other hand, identifying $\sigma(R)\sigma(\theta) = \sigma(2)$ (i.e., giving the complex orientation to the disc $\Rbar(A,0,0)$), we obtain the identification
$$\S(H^3_{\langle k_{in}^\vee,k_{out}^\vee\rangle}) = \sigma(-2n-2k_{in}-2k_{out}-2)\sigma_{in}(\cO\cC)\sigma_{out}(\cO\cC) = \sigma(H^3_{\langle k_{in}^\vee,k_{out}^\vee\rangle})$$
(cf. \eqref{eq:sigma H3}), which allows us to define $H^3_{\langle k_{in}^\vee,k_{out}^\vee\rangle}$.

\begin{lem}
    The signs of the isomorphisms
    \begin{align*}
    \S(H^2_{k_{in}^\vee,k_{out}^\vee}) & \cong \sigma(\partial)^\vee \S(H^3_{(k_{in}-1)^\vee,k_{out}^\vee}),\\
    \S(H^2_{k_{in}^\vee,k_{out}^\vee}) &\cong \sigma(\partial)^\vee \S(H^3_{k_{in}^\vee,(k_{out}-1)^\vee}),
    \end{align*}
    induced at the codimension-$1$ boundary component $\{z^{in}_{k_{in}}=0\}$ (respectively $\{z^{out}_{k_{out}} = 0\}$), are both $-1$. 
    This justifies the signs in front of the corresponding terms in \eqref{eq:H12 vee vee}.
\end{lem}
\begin{proof}
    In a neighbourhood of such a boundary component, let $z^{in}_{k_{in}} = r^{in} \exp( i\theta^{in})$. 
    Then the identification
    $$\sigma(\cR(H^3)) \sigma(\partial) \cong \sigma(\cR(H^2))$$
    is induced by 
    \begin{align*}
        \sigma(2) \sigma(\partial) &\cong \sigma(R)\sigma(\theta)\sigma(\partial) \\
        & \cong \sigma(R)\sigma(\theta)\sigma(r_{in}) \\
        & \cong \sigma(R)\sigma(r_{in})\sigma(\theta_{in}) \\
        &\cong \sigma(\partial) \sigma(2),
    \end{align*}
    where the first isomorphism is the isomorphism introduced in the definition of $\S(H^3)$; the second arises from the identification $\sigma(\partial) = \sigma(r_{in})$ at this boundary component; the third arises from the identification $\sigma(\theta) = \sigma(N_{H^3/H^2}) = \sigma(\theta_{in})$ introduced in the definition of $\S(H^2)$; and the final one arises from the complex orientation on the $z^{in}_{k_{in}}$-plane, which goes into the isomorphism \eqref{eq:RH3 or}, together with the identification $\sigma(R) = \sigma(\partial)$, which goes into the definition of $\S(H^2)$.
    The overall sign is easily verified to be $-1$, arising from commuting $\sigma(\theta_{in})$ with $\sigma(r_{in})$. 
    The computation of the second sign is identical.
\end{proof}

The sign computations at the other boundary components involved in the proof of Lemmas \ref{lem:cardy triv} and \ref{lem:cardy nontriv} are all variations on this one, or of those from Section \ref{sec:cyc oc signs}, or those explained in \cite[Appendix C]{relfukii}.

\subsection{Cyclic open--closed map respects connections}

We have an isomorphism $\sigma(\cR(\cO\cC_*,0,0)) = \sigma(t)\sigma(\theta_*)$. 
We may identify $\cR(\cO\cC_*,0,0)$ with $\{z \in \C: 0<|z|<1\}$ by recording the position of $p_*^{int}$, if the disc is parametrized as the unit disc with $p_0^{int}$ at $0$ and $p_0^\partial$ at $-i$. 
Note that $\partial/\partial t, \partial/\partial \theta_*$ form a complex-oriented basis; so the complex orientation induces the trivialization $\sigma(t)\sigma(\theta_*) = \sigma(2)$.

We identify $\sigma(t) = \sigma(\partial)$. 
Using the complex orientation of $D_k$, and the natural identification of the orientation line of the normal bundle to $\cR(H^2_\nabla,k^\vee,s,\ell)$ inside $\cR(\cO\cC_*,s,\ell) \times D_k$ with $\sigma(\theta_*)$, we obtain an identification
$$\S(H^2_\nabla,k^\vee) = \sigma(\partial)^\vee \sigma(-2k) \sigma(\cO\cC);$$
on the other hand, identifying $\sigma(\theta_*) = \sigma(e^+)$, we obtain an identification
$$\S(H^2_\nabla,k^\wedge) = \sigma(e^+)^\vee\sigma(\partial)^\vee\sigma(-2k)\sigma(\cO\cC).$$
These allow us to define the maps $H^2_{\nabla_q}$.

We equip $\cR(H^{23}_{\nabla},k^\diamond,s,\ell)$ with the boundary orientation as a codimension-$1$ boundary stratum of $\cR(H^2_\nabla,k^\diamond,s,\ell)$. 
This allows to define the maps $H^{23}_{\nabla_q}$. 
The corresponding signs in Lemma \ref{lem:conn 2} are then $+1$ by construction; the signs of the boundary components $H^{12}$ and those involving $b^{1|1}$ are analogous to the sign verification in \cite[Lemma C.6]{relfukii}. 

We have 
$$\sigma(\cR(H^3_\nabla,k^\vee,0,0)) = \sigma(t)\sigma(\theta_*)\sigma(\arg(z_k)) \sigma(2k-2).$$
We identify $\sigma(t) = \sigma(\partial)$ and $\sigma(\theta_*)\sigma(arg(z_k)) = \sigma(2)$. 
This gives 
$$\S(H^3_\nabla,k^\vee) = \sigma(\partial)^\vee\sigma(-2k)\sigma(\cO\cC),$$
which allows us to define $H^3_{\nabla_q}$. 
One checks that the signs of the boundary components $\cR(H^{23}_\nabla,k^\vee,0,0)$, $\cR(H^{34}_\nabla,k^\vee,0,0)$, and $\cR(H^2_\nabla,(k-1)^\vee,0,0)$ are as given in \eqref{eq:H23 H34}; the sign of the boundary component corresponding to $B^{1|1}$ is computed using \cite[Lemma C.2]{relfukii}. 

We orient $\cR(H^{23}_\nabla,k^\wedge,0,0)$ and $\cR(H^{34}_\nabla,k^\wedge,0,0)$ so that their orientations agree with that of $\cR(H^3_\nabla,k^\vee,0,0)$ on their overlaps; this makes the verification of the signs in \eqref{eq:H3 wedge vee} straightforward. 

\bibliographystyle{alpha}
\bibliography{references.bib}

\end{document}